\documentclass[a4paper]{amsart}

\usepackage[all]{xy}
\usepackage{amscd,amsmath,amssymb,amsfonts}
\usepackage{amsthm}
\usepackage{fancyhdr}
\usepackage{enumerate}
\usepackage{braket}

\theoremstyle{plain}
\newtheorem{thm}[subsection]{Theorem}
\newtheorem{prop}[subsection]{Proposition}
\newtheorem{lem}[subsection]{Lemma}
\newtheorem{cor}[subsection]{Corollary}
\newtheorem{conj}[subsection]{Conjecture}

\theoremstyle{definition} 
\newtheorem{defn}[subsection]{Definition}
\newtheorem{rem}[subsection]{Remark}
\newtheorem{ex}[subsection]{Example}

\newtheorem{notation}[subsection]{Notation}

\begin{document}



\title[Bounds of the multiplicity of  abelian quotient c.i singularities]{Bounds of the multiplicity of  abelian quotient complete intersection singularities}


\author{Kohsuke Shibata}


\address{Graduate School of Mathematical Sciences, University of Tokyo, 3-8-1, Komaba, Meguro-ku,  Tokyo, 153-8914, Japan.}
\email{shibata@ms.u-tokyo.ac.jp}

\maketitle

\begin{abstract}
In this paper, we investigate  the multiplicities  and the log canonical thresholds of abelian quotient complete intersection singularities in term of the special datum. 
Moreover we give  bounds of the multiplicity of abelian quotient complete intersection singularities.
\end{abstract}

\section{Introduction}

In \cite{W}, Watanabe classified all abelian quotient complete intersection singularities.
Watanabe defined a special datum (see Section 2 for detailed definitions) in order to classify abelian quotient complete intersection singularities.
Using a special datum, he gave an upper bound of the multiplicity of abelian quotient complete intersection singularities.

\begin{thm}\label{watanabe conjecture AQ}$\mathrm(${\rm Proposition 3.1  in \cite{W}}$\mathrm)$
Let  $G$ be a finite abelian subgroup of $\mathrm{SL}(n,\mathbb C)$. 
If $R=\mathbb C[x_1,\dots,x_n]^G$ is a complete intersection, then 
$$  e(R_{\mathfrak m^G})\le 2^{n-1},$$
 where $\mathfrak m^G=(x_1,\dots, x_n)\cap R$ and $e(R_{\mathfrak m^G})$ is the Hilbert-Samuel multiplicity of the local ring $R_{\mathfrak m^G}$.
\end{thm}

The following conjecture was posed by Watanabe as a generalization of Theorem \ref{watanabe conjecture AQ}.
\begin{conj}\label{watanabe conjecture}
Let $X$ be an n-dimensional variety of locally a complete intersection with  canonical singularities. Then
$$ e(\mathcal O_{X,x})\le 2^{n-1}$$
for a closed point $x$ of $X$.
\end{conj}

In \cite{Shibata}, this conjecture was refined as follows:
\begin{conj}\label{my conjecture}
Let $X$ be an n-dimensional variety of locally a complete intersection with log canonical singularities. Then
$$ e(\mathcal O_{X,x})\le 2^{n-\lceil\mathrm{lct}(\mathfrak m_{x})\rceil}$$
for a closed point $x$ of $X$ and the equality holds if and only if $\mathrm{emb}(X,x)=2n-\lceil\mathrm{lct}(\mathfrak m_{x})\rceil$, where $\mathrm{emb}(X,x)$ is the embedding dimension of $X$ at $x$ and $\mathrm{lct}(\mathfrak m_{x})$ is the log canonical threshold of $\mathfrak m_{x}$.
\end{conj}

In \cite{Shibata}, the author gave  upper bounds of the multiplicity by functions of the log canonical threshold for  locally a complete intersection singularity.
As an application, we obtained the affirmative answer to the conjecture if the dimension of a variety is less than or equal to $32$ and the variety has canonical singularities.
\begin{thm}$\mathrm(${\rm Theorem 5.6  in \cite{Shibata}}$\mathrm)$
Let $X$ be an $n$-dimensional variety of locally a complete intersection with canonical singularities.
If $n\le 32$, then
Conjecture \ref{my conjecture} holds.
In particular, if $n\le 32$, then
Conjecture \ref{watanabe conjecture} holds.
\end{thm}

In this paper,  we study  the multiplicities  and the log canonical thresholds of abelian quotient complete intersection singularities in term of the special datum.  
We give an affirmative answer to  Conjecture \ref{my conjecture} for abelian quotient complete intersection singularities using Watanabe's classification.
\begin{thm}\label{Main Theorem A}
Let $G$ be a finite abelian subgroup of $\mathrm{SL}(n,\mathbb C)$.
If $R=\mathbb C[x_1,\dots,x_n]^G$ is a complete intersection, then 
$$e(R_{\mathfrak m^G})\le 2^{n-\lceil\mathrm{lct}(\mathfrak m^{G})\rceil}$$
and the equality holds if and only if $\mathrm{emb}(R_{\mathfrak m^G})=2n-\lceil\mathrm{lct}(\mathfrak m^{G})\rceil$, where $\mathfrak m^G=(x_1,\dots, x_n)\cap R$ and $\mathrm{emb}(R_{\mathfrak m^G})$ is the embedding dimension of the local ring $R_{\mathfrak m^G}$.

\end{thm}
Moreover we give a lower bound of the multiplicity of abelian quotient complete intersection singularities using the special datum and the log canonical threshold.

The paper is organized as follows: In Section $2$, we introduce the definition of a special datum and show some basic properties of abelian quotient complete intersection singularities in term of the special datum.
In Section $3$,  we prepare some useful and important propositions. They will play a crucial role in this paper.
In Section $4$, we give an upper bound of the multiplicity of abelian quotient complete intersection singularities.
In Section $5$, we give a lower bound of the multiplicity of abelian quotient complete intersection singularities.

\section{abelian quotient complete intersection singularities}
In this section, we recall the definition of a special datum and study the properties of  abelian quotient complete intersection singularities in term of the special datum.

\begin{defn}
Let $n\ge 1$ be an integer.
An $n$-dimensional special datum $\mathbb D=(D,w)$ is a pair consisting of a  set of non-empty subsets of $\{1, 2, \dots, n\}$ (i.e. $D\subset 2^{\{1,2,\dots, n \}}\setminus\emptyset$), together with a  function $w: D\to\mathbb N$ such that:

\begin{enumerate}
\item For each $i\in\{1, 2,\dots, n\}$ we have $\{ i\}\in D$.
\item For every pair of  $J,J'\in D$, either $J\subset J'$,  $J'\subset J$
or $J\cap J'=\emptyset$.
\item  If $J$ is a maximal element with respect to the inclusion relation
``$\subset$'', then $w(J)=1$.
\item  If $J,J'\in D$ and $J\subsetneq J'$, then $w(J)>w(J')$ and $w(J')|w(J)$.
\item  For $J_1,J_2,J\in D$ with $J_1\sqsubset J, J_2\sqsubset J$, we have $w(J_1)=w(J_2)$.
(We write $J\sqsubset J'$ if  $J\subsetneq  J'$ and if there is no element of $D$ between $J$ and $J'$.)
\end{enumerate}

\end{defn}

\begin{defn}\label{special datum ring group}
Let $\mathbb D=(D,w)$ be an $n$-dimensional special datum.
We define the ring $R_{\mathbb D}$ with the maximal ideal $\mathfrak m_{\mathbb D}$ and the group $G_{\mathbb D}$  by
$$R_{\mathbb D}:=\mathbb C[x_J^{w(J)}|\ J\in D],\ \mathrm{where}\ x_J=\prod_{j\in J}x_j,$$
$$\mathfrak m_{\mathbb D}:=(x_J^{w(J)}|\ J\in D)\ \ \ \mbox{and} $$
\[ G_{\mathbb D}:= \Braket{E,\Set{(\zeta_w,\zeta_w^{-1};i,j) |
\begin{array}{l}
\text{$J_1,J_2,J\in D, i\in J_1, j\in J_2, J_1\sqsubset J,$} \\
\text{$J_2\sqsubset J$ and $w=w(J_1)=w(J_2)$}
\end{array}}}. \]
Here $E$ denotes the identity matrix in $\mathrm{SL}(n,\mathbb C)$ and $\zeta_w$ denotes a primitive $w$-th root of unit and $(a,b;i,j)$ denotes the diagonal matrix in $\mathrm{SL}(n,\mathbb C)$ whose $(i,i)$ component is $a$ and $(j,j)$ component is $b$ and the other diagonal components are $1$.
\end{defn}

\begin{prop}$\mathrm(${\rm Proposition 1.7  in \cite{W}}$\mathrm)$\label{Watanabe invariant ring}
Let $\mathbb D=(D,w)$ be an $n$-dimensional special datum.
Then $R_{\mathbb D}=\mathbb C[x_1,\dots,x_n]^{G_{\mathbb D}}$ and $R_{\mathbb D}$ is a complete intersection.
\end{prop}

\begin{thm}$\mathrm(${\rm Main Theorem  in \cite{W}}$\mathrm)$\label{Watanabe main theorem}
Let $G$ be a finite abelian subgroup of $\mathrm{SL}(n,\mathbb C)$ and if $\mathbb C[x_1,\dots,x_n]^G$ is a complete intersection, then there exist an $n$-dimensional special datum $\mathbb D$ and $T\in\mathrm{GL}(n,\mathbb C)$  such that 
$$\mathbb C[x_1,\dots,x_n]^{TGT^{-1}}=R_{\mathbb D}\ \  \mbox{and}\ \ \  TGT^{-1}=G_{\mathbb D}.$$
\end{thm}

\begin{notation}
To illustrate a special datum $\mathbb D$, we define the graph
of $\mathbb D=(D,w)$ as follows;
\begin{enumerate}
\item We represent $J\in D$ by  a circle and we write the integer $w(J)$ inside it.
\item If $J\sqsubset J'$, we join the corresponding circles by a line segment
in such a way that the circle corresponding $J'$ lies above that of $J$.
\end{enumerate}

\end{notation}

The following proposition is  immediate from Definition \ref{special datum ring group}.

\begin{prop}\label{embedding dimension}
Let $\mathbb D=(D,w)$ be an $n$-dimensional special datum.
Then $$\mathrm{emb}(R_{\mathbb D})=|\{J\in D\}|=n+|\{J\in D||J|\ge2\}|.$$
\end{prop}

\begin{ex}
Let $D=\{ \{1,\dots,n\},\{1\},\dots,\{n\}\}$ and $w$ be the function from $D$ to $\mathbb N$ such that $w(\{1\})=\cdots =w(\{n\})=a$, $w(\{1,\dots,n\})=1$.
Then $\mathbb D=(D,w)$ is a special datum.
By the definitions, we have  
$$R_{\mathbb D}=\mathbb C[x_1^a,\dots,x_n^a,x_1\cdots x_n],$$
$$G_{\mathbb D}=\Braket{(\zeta_a,\zeta_a^{-1};1,2),\dots,(\zeta_a,\zeta_a^{-1};1,n)}$$
and the graph of $\mathbb D$ is
$$
    
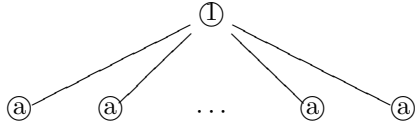
\xymatrix{
        &&{\textcircled{1}}  \ar@{-}[dll] \ar@{-}[dl]\ar@{-}[dr] \ar@{-}[drr]\\
*{\textcircled{a}}  & *{\textcircled{a}}  & {\cdots}  & *{\textcircled{a}}  & *{\textcircled{a}}  }
$$

\end{ex}

\begin{ex}
Let $D=\{ \{1,2\},\{3,4\},\{1\},\{2\},\{3\},\{4\}\}$ and $w$ be the function from $D$ to $\mathbb N$ such that $w(\{1\})=w(\{2\})=a$, $w(\{3\})=w(\{4\})=b$, $w(\{1,2\})=w(\{3,4\})=1$.
Then $\mathbb D=(D,w)$ is a special datum.
By the definitions, we have  
$$R_{\mathbb D}=\mathbb C[x_1^a,x_2^a,x_3^b,x_4^b,x_1x_2,x_3x_4],$$
$$G_{\mathbb D}=\Braket{(\zeta_a,\zeta_a^{-1};1,2),(\zeta_b,\zeta_b^{-1};3,4)}$$
and the graph of $\mathbb D$ is
$$
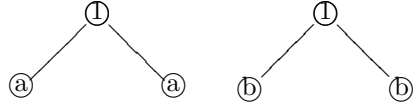
\begin{xy}
\ar @{-}(5,30) *{\textcircled{a}};(15,40) *{\textcircled{1}}
\ar @{-}(25,30) *{\textcircled{a}};(15,40) *{\bigcirc}
\ar @{-}(35,30)  *{\textcircled{b}};(45,40) *{\textcircled{1}}
\ar @{-}(55,30)  *{\textcircled{b}};(45,40) *{\bigcirc}
\end{xy}
$$
\end{ex}

\begin{defn}
A special datum $\mathbb D=(D,w)$ is said to be connected if $D$ has the unique maximal element.
\end{defn}

The following lemma is proved in the proof of Proposition 1.7 in \cite{W}

\begin{lem}\label{reduce lemma}
Let $\mathbb D=(D,w)$ be a  special datum  and $J\in D$ be a maximal element of $D$ with $|J|\ge 2$. 
Let $J_1,\dots,J_m$ be the elements of $D$ such that $J=J_1\cup\dots\cup J_m$ and
$J_i\sqsubset  J$ for $i=1,\dots,m$.
Let  $D'=D\setminus J$ and   $w': D'\to\mathbb N$ be the function such that
\begin{equation*}
w'(J')  =
\begin{cases}
w(J')/w(J_i)  & \text{if $J'\subset J_i$,}\\
w(J') & \text{if $J'\nsubseteq J$.}
\end{cases}
\end{equation*}
Then  $\mathbb D'=(D',w')$ is a special datum and 
$R_\mathbb D\cong R_{{\mathbb D}'}[Y]/(Y^{w(J_i)}-x_{J_1}\cdots x_{J_m})$.
\end{lem}

\begin{defn}
Let $\mathbb D=(D,w)$ be a  special datum  and $J\in D$ be a maximal element of $D$ with $|J|\ge 2$. 
Let $J_1,\dots,J_m$ be the elements of $D$ such that $J=J_1\cup\dots\cup J_m$ and
$J_i\sqsubset  J$ for $i=1,\dots,m$.
Then we define the special datum  $\mathbb D\setminus J=(D\setminus J,v)$ by 
\begin{equation*}
v(J')  =
\begin{cases}
w(J')/w(J_i)  & \text{if $J'\subset J_i$,}\\
w(J') & \text{if $J'\nsubseteq J$.}
\end{cases}
\end{equation*}
\end{defn}

\begin{defn}
Let $\mathbb D=(D,w)$ be a special datum and $J$ be a  element of $D$.
We define the $|J|$-dimensional special datum $\mathbb D_J=(D_J,w_J)$ by $D_J=\{J'\in D|J'\subset J\}$ and $w_J(J')=\frac{w(J')}{w(J)}$ for $J'\in D_J$.
We call $\mathbb D_J$ a connected component of $\mathbb D$ if $J$ is a maximal element of $D$.

\end{defn}

We denote by $e(R_{\mathbb D})$  the  Hilbert-Samuel multiplicity of the local ring $(R_{\mathbb D})_{\mathfrak m_{\mathbb D}}$ and denote by $\mathrm{emb}(R_{\mathbb D})$  the  embedding dimension of the local ring $(R_{\mathbb D})_{\mathfrak m_{\mathbb D}}$.
The following proposition is immediate from the definition of a connected component of a special datum.
\begin{prop}\label{non-connected}
Let $\mathbb D=(D,w)$ be a  special datum and $J_1,\dots,J_m$ be the maximal elements of $D$.
Then 
\begin{enumerate}

\item[$\mathrm{(1)}$] $R_{\mathbb D}=R_{\mathbb D_{J_1}}\otimes_{\mathbb C}\cdots\otimes_{\mathbb C}R_{\mathbb D_{J_m}}.$
\item[$\mathrm{(2)}$] $|G_{\mathbb D}|=|G_{\mathbb D_{J_1}}|\cdots|G_{\mathbb D_{J_m}}|.$
\item[$\mathrm{(3)}$] $e(R_{\mathbb D})=e(R_{\mathbb D_{J_1}})\cdots e(R_{\mathbb D_{J_m}})$. 
\item[$\mathrm{(4)}$] $\mathrm{emb}(R_{\mathbb D})=\mathrm{emb}(R_{\mathbb D_{J_1}})+\cdots+\mathrm{emb}(R_{\mathbb D_{J_m}})$.
\end{enumerate}
\end{prop}

\begin{lem}\label{a times function}
Let $\mathbb D=(D,w)$ be an $n$-dimensional connected special datum and   $J\in D$ be the maximal element of $D$.
Let $a$ be a natural number and $w^a: D\to\mathbb N$ be the function such that 
\begin{equation*}
w^{a}(J')  =
\begin{cases}
1  & \text{if $J'=J$,}\\
aw(J') & \text{if $J'\in D\setminus J$.}
\end{cases}
\end{equation*}
Then $\mathbb D^{a}=(D, w^{a})$   is a special datum with  $|G_{\mathbb D^{a}}|=a^{n-1}|G_{\mathbb D}|.$

\end{lem}

\begin{proof}
It is clear that  $\mathbb D^a$ is a special datum.
We prove $|G_{\mathbb D^{a}}|=a^{n-1}|G_{\mathbb D}|$ by induction on the dimension of $\mathbb D=(D,w)$.
If the dimension of $R_{\mathbb D}$ is $1$, then $\mathbb D=\mathbb D^{a}$ and $|G_{\mathbb D}|=|G_{\mathbb D^{a}}|$.
Therefore the lemma holds when $\mathrm{dim}R_{\mathbb D}=1$.
Now suppose that  $n\ge 2$ and this lemma holds for any special datum  of dimension at most $n-1$.

Let $J_1,\dots,J_m$ be the elements of $D$ such that $J=J_1\cup\dots\cup J_m$ and
$J_i\sqsubset  J$ for $i=1,\dots,m$.
Let $\mathbb D\setminus J=(D\setminus J,v)$,  $b\in\mathbb N$ and
$v^{b}: D\setminus J\to\mathbb N$ be the function such that 
\begin{equation*}
v^{b}(J')  =
\begin{cases}
1  & \text{if $J'=J_i$,}\\
bv(J') & \text{if $J'\in D\setminus \{J,J_1,\dots,J_m\}$.}
\end{cases}
\end{equation*}
Let $(\mathbb D\setminus J)^{b}=(D\setminus J,v^{b})$.
Note that $(\mathbb D\setminus J)^{1}=\mathbb D\setminus J$ and the dimension of a connected component of $\mathbb D\setminus J$ is less than $n$.
Therefore by the induction hypothesis and Proposition \ref{non-connected},
  $(\mathbb D\setminus J)^{b}$ is a special datum with  $$|G_{(\mathbb D\setminus J)^{b}}|=b^{n-m}|G_{{\mathbb D\setminus J}}|.$$
Choose $i_j\in J_j$ for $j=1,\dots,m$.
By the definition of $G_{\mathbb D}$, we have
\[ G_{\mathbb D}= \Braket{\Set{(\zeta_{w(J_1)},\zeta_{w(J_1)^{-1}};i_1,i_j) |
\begin{array}{l}
\text{$j=2,\dots,m$}
\end{array}},G_{({\mathbb D\setminus J})^{w(J_1)}}   }\ \mbox{and}\]
\[ G_{\mathbb D^{a}}= \Braket{\Set{(\zeta_{aw(J_1)},\zeta_{aw(J_1)^{-1}};i_1,i_j) |
\begin{array}{l}
\text{$j=2,\dots,m$}
\end{array}},G_{({\mathbb D\setminus J})^{aw(J_1)}}   }.\]

Therefore 
 $$|G_{\mathbb D^{a}}|=(aw(J_1))^{m-1}|G_{({\mathbb D\setminus J})^{aw(J_1)}}|=a^{n-1}w(J_1)^{m-1}|G_{({\mathbb D\setminus J})^{w(J_1)}}|=a^{n-1}|G_{\mathbb D}|.$$

\end{proof}

\begin{lem}\label{group reduction}
Let $n\ge 2$,  $\mathbb D=(D,w)$ be an $n$-dimensional connected special datum,   $J\in D$ be the maximal element of $D$ and
$J'$ be a maximal element of $D\setminus J$.
Then
$$|G_{\mathbb D}|=w(J')^{n-1}|G_{\mathbb D\setminus J}|.$$
\end{lem}

\begin{proof}
Let $J_1,\dots,J_m$ be the elements of $D$ such that $J=J_1\cup\dots\cup J_m$ and
$J_i\sqsubset  J$ for $i=1,\dots,m$.
Note that $w(J')=w(J_1)=\cdots=w(J_m)$.
Let $v^{w(J_1)}: D\setminus J\to\mathbb N$ be the function such that 
\begin{equation*}
v^{w(J_1)}(J')  =
\begin{cases}
1  & \text{if $J'=J_i$,}\\
w(J_1)v(J') & \text{if $J'\in D\setminus \{J,J_1,\dots,J_m\}$.}
\end{cases}
\end{equation*}
Let $(\mathbb D\setminus J)^{w(J_1)}=(D\setminus J,v^{w(J_1)})$.
By Lemma \ref{a times function}, $(\mathbb D\setminus J)^{w(J_1)}$ is a special datum with
$|G_{(\mathbb D\setminus J)^{w(J_1)}}|={w(J_1)}^{n-m}|G_{{\mathbb D\setminus J}}|.$
Choose $i_j\in J_j$ for $j=1,\dots,m$.
By the definition of $G_{\mathbb D}$, we have
\[ G_{\mathbb D}= \Braket{\Set{(\zeta_{w(J_1)},\zeta_{w(J_1)^{-1}};i_1,i_j) |
\begin{array}{l}
\text{$j=2,\dots,m$}
\end{array}},G_{({\mathbb D\setminus J})^{w(J_1)}}}.\]
Therefore
 $$|G_{\mathbb D}|={w(J_1)}^{m-1}|G_{({\mathbb D\setminus J})^{w(J_1)}}|={w(J_1)}^{n-1}|G_{\mathbb D\setminus J}|.$$
\end{proof}

We recall the definitions of singularities of pairs and the log canonical threshold.
Let $X$ be a $\mathbb Q$-Gorenstein normal variety over $\mathbb C$,
$\mathfrak a \subset \mathcal O_X$ an ideal sheaf  and $t\ge0$ a real number.  
Let $f:Y\to X$ be a resolution of singularities such that the ideal sheaf $\mathfrak a\mathcal O_Y=\mathcal O_Y(-F)$ is invertible and $\mathrm{Supp}F\cup\mathrm{Exc}(f)$ is a simple normal crossing divisor, where $\mathrm{Exc}(f)$ is the exceptional locus of $f$.
Let $K_X$ and $K_Y$ denote the canonical divisors of $X$ and $Y$, respectively.
Then there are finitely  many irreducible
(not necessarily exceptional) divisors $E_i$ on $Y$ and real numbers $a(E_i;X,\mathfrak a^t)$ so that there exists an
$\mathbb Q$-linear equivalence of $\mathbb R$-divisors
$$K_Y-f^*K_X-tF=\sum_ia(E_i;X,\mathfrak a^t)E_i.$$

\begin{defn}
Under the notation as above:
\begin{enumerate}
\item We say that the pair $(X, \mathfrak a^t)$ is canonical if $a(E_i;X,\mathfrak a^t)\ge 0$ for all exceptional $E_i$.
\item We say that the pair $(X, \mathfrak a^t)$ is log canonical if $a(E_i;X,\mathfrak a^t)\ge-1$ for all $E_i$.
\item  Suppose that $(X, \mathcal O_X)$ is log canonical.  Then we
define the log canonical threshold of $\mathfrak a$ to be
$$\mathrm{lct}(X,\mathfrak a)=\mathrm{sup}\{t\in\mathbb R_{\ge 0}|\ \mathrm{the\ pair}\ (X,\mathfrak a^t)\ \mathrm{is\ log\ canonical}\}.$$
If there is not risk of confusion, we shall simply write $\mathrm{lct}(\mathfrak a)$ instead of $\mathrm{lct}(X,\mathfrak a)$ and $\mathrm{lct}(R,\mathfrak a)$ instead of $\mathrm{lct}(\mathrm{Spec}R,\mathfrak a)$ for a ring $R$.
\item We define the multiplier ideal of $\mathfrak a$ with coefficient $t$ to be $$\mathcal J(X,\mathfrak a^t)=f_*(\lceil\sum_ia(E_i;X,\mathfrak a^t)E_i\rceil).$$
\end{enumerate}

\end{defn}

\begin{rem}
For a special datum $\mathbb D$, $(\mathrm{Spec}R_{\mathbb D}, \mathcal O_{\mathrm{Spec}R_{\mathbb D}})$ is canonical by Proposition 5.20 and Corollary 5.24 in \cite{KM}. 
\end{rem}

\begin{defn}
Let $\mathfrak a\subset \mathbb C[x_1,\dots,x_n]$ be a monomial ideal.
The Newton Polygon $\mathrm{Newt}(\mathfrak a)$ of $\mathfrak a$ is defined to be the covex hull of $\{(a_1,\dots,a_n)\in \mathbb Z^n|\ x_1^{a_1}\cdots x_n^{a_n}\in \mathfrak a\}$ in $\mathbb R^n$.

\end{defn}

Howald gave a formula computing the multiplier ideal of a monomial ideal  in \cite{How}.
\begin{thm}$\mathrm(${\rm Main Theorem  in \cite{How}}$\mathrm)$\label{toric multiplier ideal}
Let $X=\mathrm{Spec} \mathbb C[x_1,\dots,x_n]$ and $\mathfrak a\subset \mathcal O_X$ be a monomial ideal.
Then for $t\in \mathbb R_{>0}$,
$$\mathcal J(X,\mathfrak a^t)=\langle x^m|\ m+(1,\dots,1)\in \mathrm{interior}\ \mathrm{of}  \ t\mathrm{Newt}(\mathfrak a)\rangle,$$
where $t\mathrm{Newt}(\mathfrak a)=\{t{\bf a}\in\mathbb R^n|\ {\bf a}\in \mathrm{Newt}(\mathfrak a)\}$. 
\end{thm}

The following corollary is an immediate consequence of the definition of a log canonical threshold and Theorem\ref{toric multiplier ideal}.
\begin{cor}$\mathrm(${\rm Example 5  in \cite{How}}$\mathrm)$\label{toric multiplier ideal log canonical threshold}
Let $X=\mathrm{Spec} \mathbb C[x_1,\dots,x_n]$ and $\mathfrak a\subset \mathcal O_{X}$ be a monomial ideal.
Then
$$\mathrm{lct}(X,\mathfrak a)=\mathrm{sup}\{t\in\mathbb R_{>0}|\ (1,\dots,1)\in t\mathrm{Newt}(\mathfrak a)\}.$$ 
\end{cor}

\begin{lem}\label{log canonical threshold lemma}
Let $\mathbb D=(D,w)$ be an $n$-dimensional  special datum.
Let $\mathfrak a_\mathbb D\subset \mathbb C[x_1,\dots,x_n]$ be the ideal generated by $x_J^{w(J)}$ for $J\in D$.
Then $$\mathrm{lct}(\mathfrak m_{{\mathbb D}})=\mathrm{lct}(\mathfrak a_{{\mathbb D}})=\mathrm{sup}\{t\in\mathbb R_{>0}|\ (1,\dots,1)\in t\mathrm{Newt}(\mathfrak a_{\mathbb D})\}.$$
\end{lem}
\begin{proof}
It follows from Proposition 5.20 in \cite{KM} that for $t\in \mathbb R_{>0}$, the pair  $(\mathrm{Spec}R_{\mathbb D}, \mathfrak m_{\mathbb D}^t)$  is log canonical
if and only if the pair $(\mathrm{Spec} \mathbb C[x_1,\dots,x_n],\mathfrak a_{\mathbb D}^t)$ is log canonical.
Therefore $\mathrm{lct}(\mathfrak m_{{\mathbb D}})=\mathrm{lct}(\mathfrak a_{{\mathbb D}})$.
By Corollary \ref{toric multiplier ideal log canonical threshold}, we have $$\mathrm{lct}(\mathfrak m_{{\mathbb D}})=\mathrm{lct}(\mathfrak a_{{\mathbb D}})=\mathrm{sup}\{t\in\mathbb R_{>0}|\ (1,\dots,1)\in t\mathrm{Newt}(\mathfrak a_{\mathbb D})\}.$$
\end{proof}

\begin{prop}\label{log canonical threshold}
Let $\mathbb D=(D,w)$ be an $n$-dimensional  special datum and $J_1,\dots,J_m$ be the maximal elements of $D$.
Then 
$$\mathrm{lct}(\mathfrak m_{{\mathbb D}})=\mathrm{lct}(\mathfrak m_{\mathbb D_{J_1}})+\cdots+\mathrm{lct}(\mathfrak m_{\mathbb D_{J_m}}).$$

\end{prop}

\begin{proof}
Let $n_0=0$,  $n_i$ be the dimension of $R_{\mathbb D_{J_i}}$ for $i=1,\dots,m$ and $N_{i}=n_0+\cdots+n_{i}$ for $i=0,\dots,m-1$. 
Then we may assume that $J_i=\{N_{i-1}+1,\dots,N_{i-1}+n_i\}$ for $i=1,\dots,m$.
Let $\mathfrak a_{\mathbb D_{J_i}}\subset \mathbb C[x_{N_{i-1}+1},\dots,x_{N_{i-1}+n_i}]$ be the ideal generated by $x_J^{w(J)}$ for $J\in D_{J_i}$.
By Lemma \ref{log canonical threshold lemma},
$$\mathrm{lct}(\mathfrak m_{\mathbb D})=\mathrm{lct}(\mathfrak a_{\mathbb D})=\mathrm{sup}\{t\in\mathbb R_{>0}|\ (1,\dots,1)\in t\mathrm{Newt}(\mathfrak a_{\mathbb D})\subset\mathbb R^{n}\}\ \ \mbox{and}$$
$$\mathrm{lct}(\mathfrak m_{{\mathbb D_{J_i}}})=\mathrm{lct}(\mathfrak a_{{\mathbb D_{J_i}}})=\mathrm{sup}\{t\in\mathbb R_{>0}|\ (1,\dots,1)\in t\mathrm{Newt}(\mathfrak a_{\mathbb D_{J_i}})\subset\mathbb R^{n_i}\}.$$
Since $\mathfrak a_\mathbb D=\sum_i\mathfrak a_{\mathbb D_{J_i}}\mathbb C[x_1,\dots,x_{n}]$,
we have $\mathrm{lct}(\mathfrak a_{{\mathbb D}})=\mathrm{lct}(\mathfrak a_{\mathbb D_{J_1}})+\cdots+\mathrm{lct}(\mathfrak a_{\mathbb D_{J_m}}).$
Therefore we have $\mathrm{lct}(\mathfrak m_{{\mathbb D}})=\mathrm{lct}(\mathfrak m_{\mathbb D_{J_1}})+\cdots+\mathrm{lct}(\mathfrak m_{\mathbb D_{J_m}}).$

\end{proof}

\begin{prop}\label{log canonical threshold max 1 lct}
Let $\mathbb D=(D,w)$ be an $n$-dimensional connected special datum, $J\in D$ be the maximal element of $D$ with $|J|\ge 2$ and $J'$ be a maximal element of $D\setminus J$.
Then $$\mathrm{lct}(\mathfrak m_{{\mathbb D}})=\mathrm{max}\Bigl\{1, \frac{\mathrm{lct}(\mathfrak m_{\mathbb D\setminus J})}{w(J')}\Bigl\}.$$
\end{prop}

\begin{proof}
Let $\mathfrak a_\mathbb D\subset \mathbb C[x_1,\dots,x_n]$ be the ideal generated by $x_J^{w(J)}$ for $J\in D$.
By Lemma \ref{log canonical threshold lemma}, we have $$\mathrm{lct}(\mathfrak m_{{\mathbb D}})=\mathrm{lct}(\mathfrak a_{{\mathbb D}})=\mathrm{sup}\{t\in\mathbb R_{>0}|\ (1,\dots,1)\in t\mathrm{Newt}(\mathfrak a_{\mathbb D})\}.$$
Let $\mathbb D\setminus J=(D\setminus J,v)$ and  $a=w(J')$.
Let $\mathfrak a_{\mathbb D\setminus J}\subset \mathbb C[x_1,\dots,x_n]$ be the ideal generated by $x_J^{w(J)}$ for $J\in D\setminus J$ and $\mathfrak b_{\mathbb D\setminus J}\subset \mathbb C[x_1,\dots,x_n]$ be the ideal generated by $x_J^{v(J)}$ for $J\in D\setminus J$.
Then by  Corollary \ref{toric multiplier ideal log canonical threshold} and Lemma \ref{log canonical threshold lemma}, 
$$\mathrm{lct}(\mathfrak a_{{\mathbb D\setminus J}})=\mathrm{sup}\{t\in\mathbb R_{>0}|\ (1,\dots,1)\in t\mathrm{Newt}(\mathfrak a_{\mathbb D\setminus J})\}\ \ \mathrm{and}$$
$$\mathrm{lct}(\mathfrak m_{{\mathbb D\setminus J}})=\mathrm{lct}(\mathfrak b_{{\mathbb D\setminus J}})=\mathrm{sup}\{t\in\mathbb R_{>0}|\ (1,\dots,1)\in t\mathrm{Newt}(\mathfrak b_{\mathbb D\setminus J})\}.$$
Since $\mathrm{Newt}(\mathfrak a_{\mathbb D\setminus J})=a\mathrm{Newt}(\mathfrak b_{\mathbb D\setminus J})$, we have  $$\mathrm{lct}(\mathfrak m_{{\mathbb D\setminus J}})=\mathrm{lct}(\mathfrak b_{{\mathbb D\setminus J}})=a{\mathrm{lct}(\mathfrak a_{{\mathbb D\setminus J}})}.$$
Since $\mathfrak a_{\mathbb D}=(x_1\cdots x_n)+\mathfrak a_{\mathbb D\setminus J}$, we have 
$$(1,\dots,1)\in\mathrm{Newt}(\mathfrak a_{\mathbb D})=\mathrm{Newt}((x_1\cdots x_n)+\mathfrak a_{\mathbb D\setminus J}).$$
Therefore we have $\mathrm{lct}(\mathfrak a_{{\mathbb D}})\ge 1$.
If $(1,\dots,1)\in t\mathrm{Newt}(\mathfrak a_{\mathbb D})$ for $t>1$, then $(1,\dots,1)\in t\mathrm{Newt}(\mathfrak a_{\mathbb D\setminus J})$.
This implies that if $\mathrm{lct}(\mathfrak m_{{\mathbb D}})=\mathrm{lct}(\mathfrak a_{{\mathbb D}})>1$, then 
$$\mathrm{lct}(\mathfrak a_{{\mathbb D}})=\mathrm{lct}(\mathfrak a_{{\mathbb D\setminus J}})=\frac{\mathrm{lct}(\mathfrak b_{{\mathbb D\setminus J}})}{a}=\frac{\mathrm{lct}(\mathfrak m_{{\mathbb D\setminus J}})}{a}.$$
Therefore  $$\mathrm{lct}(\mathfrak m_{{\mathbb D}})=\mathrm{max}\Bigl\{1, \frac{\mathrm{lct}(\mathfrak m_{\mathbb D\setminus J})}{w(J')}\Bigl\}.$$
\end{proof}

\begin{cor}
Let $\mathbb D=(D,w)$ be a  special datum.
Then $$\mathrm{lct}(\mathfrak m_{{\mathbb D}})\ge |\{J\in D|\ J {\rm \ is\ a\ maximal\ element\ of\ }D\}|.$$
\end{cor}

\begin{proof}
This follows immediately from Proposition \ref{log canonical threshold} and Proposition \ref{log canonical threshold max 1 lct}.
\end{proof}

\begin{prop}
\label{embedding dimension log canonical threshold1}
Let $\mathbb D=(D,w)$ be an $n$-dimensional  connected special datum and
 $J\in D$ be the maximal element of $D$ with $|J|\ge 2$.
Let $J_1,\dots,J_m$ be the elements of $D$ such that $J=J_1\cup\dots\cup J_m$ and
$J_i\sqsubset  J$ for $i=1,\dots,m$.
Then $$\mathrm{emb}(R_{\mathbb D})\le 2n-\lceil\mathrm{lct}(\mathfrak m_{{\mathbb D}_{J_1}})\rceil-\cdots-\lceil\mathrm{lct}(\mathfrak m_{{\mathbb D}_{J_m}})\rceil+1\le 2n-\lceil\mathrm{lct}(\mathfrak m_{\mathbb D})\rceil.$$
Moreover, the following are equvalent:
\begin{enumerate}
\item[\rm{(1)}]
 $\mathrm{emb}(R_{\mathbb D})=2n-\lceil\mathrm{lct}(\mathfrak m_{\mathbb D})\rceil$.

\item[\rm{(2)}]
$\mathrm{emb}(R_{{\mathbb D_{J_i}}})=2\mathrm{dim}R_{{\mathbb D_{J_i}}}-\lceil\mathrm{lct}(\mathfrak m_{{\mathbb D_{J_i}}})\rceil$ for $i=1,\dots ,m$ and\\
$\lceil\mathrm{lct}(\mathfrak m_{{\mathbb D_{J_1}}})\rceil+\dots+\lceil\mathrm{lct}(\mathfrak m_{{\mathbb D_{J_m}}})\rceil-1=\lceil\mathrm{lct}(\mathfrak m_{{\mathbb D}})\rceil$.
\end{enumerate}
\end{prop}

\begin{proof}
We prove this by induction on the dimension of $\mathbb D=(D,w)$.
Note that if the dimension of $R_{\mathbb D}$ is $1$, then $R_{\mathbb D}=\mathbb C[x_1]$, $\mathrm{emb}(R_{\mathbb D})=1$ and $\mathrm{lct}(\mathfrak m_{\mathbb D})=1$
and if the dimension of $R_{\mathbb D}$ is $2$, then $R_{\mathbb D}=\mathbb C[x_1^{w(1)},x_2^{w(2)},x_1x_2]$, $\mathrm{emb}(R_{\mathbb D})=3$ and $\mathrm{lct}(\mathfrak m_{\mathbb D})=1$ by Proposition \ref{log canonical threshold max 1 lct}.
Therefore the proposition holds when $\mathrm{dim}R_{\mathbb D}=2$.
Now suppose that  $n\ge 3$ and this proposition  holds for any special datum  of dimension at most $n-1$.

Let $a=w(J_1)$.
Note that  the dimension of a connected component of $\mathbb D\setminus J$ is less than $n$ and
 $\mathrm{emb}(R_{\mathbb D})=\mathrm{emb}(R_{\mathbb D\setminus J})+1$.
 By the induction hypothesis
\begin{align*}
\mathrm{emb}(R_{\mathbb D})&= \mathrm{emb}(R_{\mathbb D\setminus J})+1\\
&=\mathrm{emb}(R_{{\mathbb D_{J_1}}})+\cdots+\mathrm{emb}(R_{{\mathbb D_{J_m}}})+1\\
&\le 2n-\lceil\mathrm{lct}(\mathfrak m_{{\mathbb D}_{J_1}})\rceil-\cdots-\lceil\mathrm{lct}(\mathfrak m_{{\mathbb D}_{J_m}})\rceil+1.
\end{align*}
 Note that $a\ge 2$, $m\ge 2$, $\mathrm{lct}(\mathfrak m_{{\mathbb D\setminus J}})=\mathrm{lct}(\mathfrak m_{{\mathbb D}_{J_1}})+\cdots+\mathrm{lct}(\mathfrak m_{{\mathbb D}_{J_m}})$ by  Proposition \ref{log canonical threshold} and $\mathrm{lct}(\mathfrak m_{{\mathbb D}})=\mathrm{max}\Bigl\{1, \frac{\mathrm{lct}(\mathfrak m_{\mathbb D\setminus J})}{w(J_1)}\Bigl\}$ by Proposition \ref{log canonical threshold max 1 lct}.

If $\mathrm{lct}(\mathfrak m_{{\mathbb D}})=1$, then 
\begin{align*}
\lceil\mathrm{lct}(\mathfrak m_{{\mathbb D}})\rceil=1\le \lceil\mathrm{lct}(\mathfrak m_{{\mathbb D}_{J_1}})\rceil+\cdots+\lceil\mathrm{lct}(\mathfrak m_{{\mathbb D}_{J_m}})\rceil-1.
\end{align*}

If $\mathrm{lct}(\mathfrak m_{{\mathbb D}})=\frac{\mathrm{lct}(\mathfrak m_{\mathbb D\setminus J})}{a}$, then 
\begin{align*}
\lceil\mathrm{lct}(\mathfrak m_{{\mathbb D}})\rceil&=\Big{\lceil}\frac{\mathrm{lct}(\mathfrak m_{{\mathbb D\setminus J}})}{a}\Big{\rceil}\\
&\le\Big{\lceil}\frac{\mathrm{lct}(\mathfrak m_{{\mathbb D}_{J_1}})+\cdots+\mathrm{lct}(\mathfrak m_{{\mathbb D}_{J_m}})}{2}\Big{\rceil}\\
&\le\lceil\mathrm{lct}(\mathfrak m_{{\mathbb D}_{J_1}})\rceil+\cdots+\lceil\mathrm{lct}(\mathfrak m_{{\mathbb D}_{J_m}})\rceil-1.
\end{align*}

Therefore the proposition holds.
\end{proof}

\begin{rem}
Even if a complete intersection singularity $(R, \mathfrak m)$ is not an invariant ring, the inequality $\mathrm{emb}(R)\le 2n-\lceil\mathrm{lct}(\mathfrak m)\rceil$ holds.
See Corollary 4.5 and Lemma 5.1 in \cite{Shibata}.
\end{rem}

\begin{lem}
\label{embedding dimension log canonical threshold2}
Let $\mathbb D=(D,w)$ be an $n$-dimensional  connected special datum and
 $J\in D$ be the maximal element of $D$ with $|J|\ge 2$.
Let $J_1,\dots,J_m$ be the elements of $D$ such that $J=J_1\cup\dots\cup J_m$ and
$J_i\sqsubset  J$ for $i=1,\dots,m$.
If $$\mathrm{lct}(\mathfrak m_{{\mathbb D}})=\frac{\mathrm{lct}(\mathfrak m_{\mathbb D\setminus J})}{w(J_1)}\ \ \mbox{and}$$ 
$$\lceil\mathrm{lct}(\mathfrak m_{{\mathbb D_{J_1}}})\rceil+\dots+\lceil\mathrm{lct}(\mathfrak m_{{\mathbb D_{J_m}}})\rceil-1=\lceil\mathrm{lct}(\mathfrak m_{{\mathbb D}})\rceil,$$ then $w(J_1)=2$ and $\lceil\mathrm{lct}(\mathfrak m_{{\mathbb D\setminus J}})\rceil-\lceil\mathrm{lct}(\mathfrak m_{\mathbb D})\rceil=1$.
\end{lem}

\begin{proof}
If $\lceil\mathrm{lct}(\mathfrak m_{{\mathbb D}})\rceil\ge3$, then 
$\lceil\mathrm{lct}(\mathfrak m_{{\mathbb D}})\rceil<\lceil\mathrm{lct}(\mathfrak m_{{\mathbb D}_{J_1}})\rceil+\cdots+\lceil\mathrm{lct}(\mathfrak m_{{\mathbb D}_{J_m}})\rceil-1.$
Indeed, by  Proposition \ref{log canonical threshold}
\begin{align*}
\lceil\mathrm{lct}(\mathfrak m_{{\mathbb D}})\rceil&<\lceil2\mathrm{lct}(\mathfrak m_{{\mathbb D}})\rceil-1\le\lceil\mathrm{lct}(\mathfrak m_{\mathbb D\setminus J})\rceil-1\\
&\le\lceil\mathrm{lct}(\mathfrak m_{{\mathbb D}_{J_1}})\rceil+\cdots+\lceil\mathrm{lct}(\mathfrak m_{{\mathbb D}_{J_m}})\rceil-1.
\end{align*}
Therefore we have that 
$\lceil\mathrm{lct}(\mathfrak m_{{\mathbb D}})\rceil$ is $1$ or $2$.\\
Note that $m\ge 2$ and $\mathrm{lct}(\mathfrak m_{{\mathbb D}})=\mathrm{max}\Bigl\{1, \frac{\mathrm{lct}(\mathfrak m_{\mathbb D\setminus J})}{w(J_1)}\Bigl\}$ by Proposition \ref{log canonical threshold max 1 lct}.
\begin{enumerate}
\item
If $\lceil\mathrm{lct}(\mathfrak m_{{\mathbb D}})\rceil=1$, then $m=2,\lceil\mathrm{lct}(\mathfrak m_{{\mathbb D}_{J_1}})\rceil=\lceil\mathrm{lct}(\mathfrak m_{{\mathbb D}_{J_2}})\rceil=1$.
Therefore $\mathrm{lct}(\mathfrak m_{{\mathbb D}})=\mathrm{lct}(\mathfrak m_{{\mathbb D}_{J_1}})=\mathrm{lct}(\mathfrak m_{{\mathbb D}_{J_2}})=1$.
By Proposition \ref{log canonical threshold}, we have $\mathrm{lct}(\mathfrak m_{{\mathbb D\setminus J}})=2$.
Since $\mathrm{lct}(\mathfrak m_{{\mathbb D}})=\frac{\mathrm{lct}(\mathfrak m_{\mathbb D\setminus J})}{w(J_1)}$, we have $w(J_1)=2$ and $\lceil\mathrm{lct}(\mathfrak m_{{\mathbb D\setminus J}})\rceil-\lceil\mathrm{lct}(\mathfrak m_{\mathbb D})\rceil=1$.

\item
If $\lceil\mathrm{lct}(\mathfrak m_{{\mathbb D}})\rceil=2$, then $m=2$ or $m=3$.

We assume that $m=2$.
Then we may assume that  $\lceil\mathrm{lct}(\mathfrak m_{{\mathbb D}_{J_1}})\rceil=1, \lceil\mathrm{lct}(\mathfrak m_{{\mathbb D}_{J_2}})\rceil=2$.
By Proposition \ref{log canonical threshold}, we have $\mathrm{lct}(\mathfrak m_{{\mathbb D\setminus J}})\le 3$.
Since $\mathrm{lct}(\mathfrak m_{{\mathbb D}})=\frac{\mathrm{lct}(\mathfrak m_{\mathbb D\setminus J})}{w(J_1)}$, we have $w(J_1)=2$ and $\lceil\mathrm{lct}(\mathfrak m_{{\mathbb D\setminus J}})\rceil-\lceil\mathrm{lct}(\mathfrak m_{\mathbb D})\rceil=1$.

We assume that $m=3$. Then $\lceil\mathrm{lct}(\mathfrak m_{{\mathbb D}_{J_1}})\rceil=\cdots=\lceil\mathrm{lct}(\mathfrak m_{{\mathbb D}_{J_3}})\rceil=1$.
By Proposition \ref{log canonical threshold}, we have $\mathrm{lct}(\mathfrak m_{{\mathbb D\setminus J}})=3$.
Since $\mathrm{lct}(\mathfrak m_{{\mathbb D}})=\frac{\mathrm{lct}(\mathfrak m_{\mathbb D\setminus J})}{w(J_1)}$, we have $w(J_1)=2$ and $\lceil\mathrm{lct}(\mathfrak m_{{\mathbb D\setminus J}})\rceil-\lceil\mathrm{lct}(\mathfrak m_{\mathbb D})\rceil=1$.
\end{enumerate}
Therefore this lemma holds.
\end{proof}


\section{relation between $e(R_{\mathbb D})$ and $e(R_{\mathbb D\setminus J})$}

In this section,  we investigate the relation between $e(R_{\mathbb D})$ and $e(R_{\mathbb D\setminus J}).$

\begin{prop}\label{lct=lct}
Let $\mathbb D=(D,w)$ be an $n$-dimensional connected special datum and $J\in D$ be the maximal element of $D$ with $|J|\ge 2$.
Let $J_1,\dots,J_m$ be the elements of $D$ such that $J=J_1\cup\dots\cup J_m$ and
$J_i\sqsubset  J$ for $i=1,\dots,m$.
Then $$e(R_{\mathbb D})\le w(J_1)e(R_{\mathbb D\setminus J}).$$
Moreover, if $\mathrm{lct}(\mathfrak m_{{\mathbb D}})=\frac{\mathrm{lct}(\mathfrak m_{\mathbb D\setminus J})}{w(J_1)}$, then $$e(R_{\mathbb D})=w(J_1)e(R_{\mathbb D\setminus J}).$$
\end{prop}

\begin{proof}
First we prove that $e(R_{\mathbb D})\le w(J_1)e(R_{\mathbb D\setminus J}).$
Let  $a=w(J_1)$.
Note that  by Lemma \ref{reduce lemma},
\begin{align*}
R_\mathbb D&=\mathbb C[x_{J'}^{w(J')}|J'\in D\setminus J][x_J]\\
&\cong \mathbb C[x_{J'}^{w(J')}|J'\in D\setminus J][Y]/(Y^a-x_{J_1}^a\cdots x_{J_m}^a)\\
&\cong R_{{\mathbb D\setminus J}}[Y]/(Y^a-x_{J_1}\cdots x_{J_m}).
\end{align*}
Let $\mathfrak m=(x_{J'}^{w(J')}|J'\in D\setminus J)\subset \mathbb C[x_{J'}^{w(J')}|J'\in D\setminus J]$ and  $I$ be  a minimal reduction of the maximal ideal  of $\mathbb C[x_{J'}^{w(J')}|J'\in D\setminus J]_{\mathfrak m}$.
Note that $I(R_\mathbb D)_{\mathfrak m_\mathbb D}$ is an ${\mathfrak m_\mathbb D}$-primary ideal of $(R_\mathbb D)_{\mathfrak m_\mathbb D}$.
Then 
\begin{align*}
e(R_\mathbb D) & \le e(I(R_\mathbb D)_{\mathfrak m_\mathbb D}) \\
&=\ell\bigl((R_\mathbb D)_{\mathfrak m_\mathbb D}/I(R_\mathbb D)_{\mathfrak m_\mathbb D}\bigr)\\
 &=\ell\bigl((\mathbb C[x_{J'}^{w(J')}|J'\in D\setminus J]_{\mathfrak m}/I)[Y]/(Y^a-x_{J_1}^a\cdots x_{J_m}^a)\bigr)\\
& =a\ell(\mathbb C[x_{J'}^{w(J')}|J'\in D\setminus J]_{\mathfrak m}/I)\\
&=ae(\mathbb C[x_{J'}^{w(J')}|J'\in D\setminus J]_{\mathfrak m})\\
&=ae(R_{\mathbb D\setminus J}).
\end{align*}
The first equality holds by  Proposition 11.1.10 in \cite{HS}.

Next we prove that if $\mathrm{lct}(\mathfrak m_{{\mathbb D}})=\frac{\mathrm{lct}(\mathfrak m_{\mathbb D\setminus J})}{w(J_1)}$, then $e(R_{\mathbb D})=w(J_1)e(R_{\mathbb D\setminus J}).$
We assume that $\mathrm{lct}(\mathfrak m_{{\mathbb D}})=\frac{\mathrm{lct}(\mathfrak m_{\mathbb D\setminus J})}{w(J_1)}$.
Let   $\boldsymbol{e}_i=(0,\cdots,0,\overset{i}{\check{1}},0,\cdots,0)\in\mathbb R^n$.
By Lemma \ref{log canonical threshold lemma}, there exist $J'_1,\dots,J'_k\in D\setminus J$ and $s,t_1,\dots,t_k\in\mathbb N$ such that 
$$\frac{t_1}{s}+\cdots+\frac{t_k}{s}=\mathrm{lct}(\mathfrak m_{\mathbb D})\ge1\ \mathrm{and}$$
$$\frac{t_1}{s}w(J'_1)\boldsymbol{e}_{J'_1}+\cdots+\frac{t_k}{s}w(J'_k)\boldsymbol{e}_{J'_k}=(1,\dots,1),$$
where $\boldsymbol{e}_{J'_i}=\sum_{j\in J'_i}\boldsymbol{e}_{j}$.
Therefore $$s\le t_1+\cdots+t_k\ \mathrm{and}$$
$$(x_1\cdots x_n)^s=(x_{J'_1}^{w(J'_1)})^{t_1}\cdots (x_{J'_k}^{w(J'_k)})^{t_k}.$$
By the definition of  the integral closure of ideals,   the integral closure of $I(R_\mathbb D)_{\mathfrak m_\mathbb D}=I\mathbb C[x_{J'}^{w(J')}|J'\in D\setminus J][x_J]_{\mathfrak m_\mathbb D}$ contains $x_J=x_1\cdots x_n$.
This implies that  the integral closure of $I(R_\mathbb D)_{\mathfrak m_\mathbb D}$ is the maximal ideal of  $(R_\mathbb D)_{\mathfrak m_\mathbb D}$.
Hence  $I(R_\mathbb D)_{\mathfrak m_\mathbb D}$ is a minimal reduction of the maximal ideal of  $(R_\mathbb D)_{\mathfrak m_\mathbb D}$ (see for example Corollary 8.3.6 and Proposition 8.3.7 in \cite{HS}).
Thus 
\begin{align*}
e(R_\mathbb D) & = \ell\bigl((R_\mathbb D)_{\mathfrak m_\mathbb D}/I(R_\mathbb D)_{\mathfrak m_\mathbb D}\bigr)
\end{align*}
(see for example Proposition 11.2.2 in \cite{HS}).
Therefore by the above discussion, we have $e(R_{\mathbb D})=w(J_1)e(R_{\mathbb D\setminus J}).$

\end{proof}

\begin{defn}
Let $(R,\mathfrak m)$ be a Noetherian local ring of dimension $n$ and $I_1,\dots,I_n$ be $\mathfrak m$-primary ideals.
Let us consider the function $H:\mathbb Z_{\ge0}^n\to\mathbb Z_{\ge 0}$ given by
$$H(r_1,\dots,r_n)=\ell(R/I_1^{r_1}\cdots I_n^{r_n})$$ 
for all $(r_1,\dots,r_n)\in\mathbb Z_{\ge0}^n$.
Then there exists a polynomial $P\in\mathbb Q[x_1,\dots,x_n]$ of degree $n$ such that
$$H(r_1,\dots,r_n)=P(r_1,\dots,r_n),$$
for all sufficiently large $r_1,\dots,r_n\in\mathbb Z_{\ge0}$ and the coefficient of the monomial $x_1\cdots x_n$ in $P$ is an integer.
The coefficient of the monomial $x_1\cdots x_n$ in $P$ is called the mixed multiplicity of $I_1,\dots,I_n$ and is denoted by $e(I_1,\dots,I_n)$.
\end{defn}

Bivi$\mathrm{\acute{a}}$-Ausina generalized the notion of mixed multiplicities and defined the invariant $\sigma_R(I_1,\dots,I_n)$ in \cite{A1}.
We use this invariant in order to prove Proposition \ref{key proposition2}.
\begin{defn}
Let $(R,\mathfrak m)$ be a Noetherian local ring of dimension $n$ and $I_1,\dots,I_n$ be  ideals of $R$.
Then we define 
$$\sigma_R(I_1,\dots,I_n)=\mathrm{max}_{r\in\mathbb Z_{\ge0}}e(I_1+\mathfrak m^r,\dots,I_n+\mathfrak m^r),$$
when the number on right-hand side is finite.
If the set $$\{e(I_1+\mathfrak m^r,\dots,I_n+\mathfrak m^r)|\ r\in\mathbb Z_{\ge0}\}$$ is non-bounded then we set
$\sigma_R(I_1,\dots,I_n)=\infty$.
\end{defn}

Let $I_1,\dots,I_c$ be  ideals of a Noetherian local ring  $(R,\mathfrak m)$ with $k=R/\mathfrak m$ an infinite field.
Let us consider a generating system $a_{i1},\dots,a_{is_i}$ of $I_i$.
Let $s=s_1+\cdots+s_c$.
We say that a property holds for sufficiently general elements of $I_1\oplus\cdots\oplus I_c$ 
if there exists a non-empty Zariski-open set $U$ in $k^s$ such that all elements $(g_1,\dots,g_c)\in I_1\oplus\cdots\oplus I_c$
satisfy the  said property provided that $g_i=\sum_{1\le j\le s_i} u_{ij}a_{ij}$, $i=1,\dots,c$, where $(u_{11},\dots,u_{1s_1},\dots,u_{c1},\dots,u_{cs_c})\in U$.

\begin{prop} $\mathrm{(Proposition\ 2.9\ in}$ \cite{A1}, $\mathrm{See\ Proposition\ 2.2\ in}$ \cite{A2}$\mathrm{)}$\label{lemma1 of A2}
Let $I_1,\dots,I_n$ be  ideals of an $n$-dimensional Noetherian local ring  $(R,\mathfrak m)$ with $R/\mathfrak m$ an infinite field.
Then $\sigma_R(I_1,\dots,I_n)<\infty$ if and only if there exist elements $g_i\in I_i$ for $i=1,\dots,n$ such that $(g_1,\dots,g_n)$ is an $\mathfrak m$-primary ideal.
In this case, we have that $\sigma_R(I_1,\dots,I_n)=e(g_1,\dots,g_n)$ for sufficiently general elements $(g_1,\dots,g_n)\in I_1\oplus\cdots\oplus I_n$, where $e(g_1,\dots,g_n)$ is the Hilbert-Samuel multiplicity of the ideal of $R$ generated by $g_1,\dots,g_n$.
\end{prop}

\begin{lem}  $\mathrm{(Corollary\ 2.5, Lemma\ 2.6 \ in}$ \cite{A2}$\mathrm{)}$\label{lemma2 of A2}
Let  $(R,\mathfrak m)$ be an $n$-dimensional Noetherian local ring with $R/\mathfrak m$ an infinite field.
Let $I_1,\dots,I_n$ be  ideals of $R$ with $\sigma_R(I_1,\dots,I_n)<\infty$.
Then \begin{enumerate}
\item[$\mathrm{(1)}$]
Let $J_1,\dots,J_n$ be ideals of $R$ such that $J_i\subset I_i$ for all $i=1,\dots,n$ and $\sigma_R(J_1,\dots,J_n)<\infty$.
Then$$\sigma_R(J_1,\dots,J_n)\ge\sigma_R(I_1,\dots,I_n).$$
\item[$\mathrm{(2)}$]
For all $r_1,\dots,r_n\in \mathbb N$, 
$$\sigma_R(I_1^{r_1},\dots,I_n^{r_n})<\infty\ \  \mbox{and}\ \ \sigma_R(I_1^{r_1},\dots,I_n^{r_n})=r_1\cdots r_n\sigma_R(I_1,\dots,I_n).$$
\end{enumerate}

\end{lem}

\begin{lem} \label{lemma2 of A2 integral closure}
Let  $(R,\mathfrak m)$ be an $n$-dimensional Noetherian local ring with $R/\mathfrak m$ an infinite field.
Let $I_1,\dots,I_n$ be  ideals of $R$ with $\sigma_R(I_1,\dots,I_n)<\infty$.
Then 
$$\sigma_R(\overline{I_1},\dots,\overline{I_n})=\sigma_R(I_1,\dots,I_n),$$
where $\overline{I_i}$ is the integral closure of $I_i$.
\end{lem}

\begin{proof}
This lemma follows from $e(I_1+\mathfrak m^r,\dots,I_n+\mathfrak m^r)=e(\overline{I_1}+\mathfrak m^r,\dots,\overline{I_n}+\mathfrak m^r)$ for any $r$ (See Theorem 17.4.9 in \cite{HS}).
\end{proof}

\begin{lem}\label{sigma multiplicity lemma}
Let $A=\mathbb C[y_1,\dots,y_{n+c},Y]_{(y_1,\dots,y_{n+c},Y)}$,
$f_1,\dots,f_{c},f$ be elements of the ideal $(y_1,\dots,y_{n+c})$ of $A$ and
 $R=\mathbb C[y_1,\dots,y_{n+c}]/(f_1,\dots,f_{c})$.
Let $a$ be a natural number and  $S=R[Y]/(Y^{a}-f)$.
We assume that $R$ and $S$ are  $n$-dimensional complete intersection rings.
Let  $\mathfrak n=(y_1,\dots,y_{n+c},Y)$ be the maximal ideal of $A$ and  $\mathfrak m=(y_1,\dots,y_{n+c},Y)$ be the maximal ideal of $S_{(y_1,\dots,y_{n+c},Y)}$.
Then 
$$e(S_\mathfrak m)=\sigma_A((f_1),\dots,(f_c),(Y^{a},f),\mathfrak n,\dots,\mathfrak n).$$

\end{lem}

\begin{proof}
By  Proposition \ref{lemma1 of A2}, 
$$\sigma_A((f_1),\dots,(f_c),(Y^{a},f),\mathfrak n,\dots,\mathfrak n)=e(f_1,\dots,f_c,sY^{a}+tf,h_1,\dots,h_n)$$
for  sufficiently general elements  $(sY^a+tf,h_1,\dots,h_n)\in(Y^a,f)\oplus\mathfrak n\oplus\cdots\oplus \mathfrak n$.
Fix such $s,t$. We may assume that $s$ and $t$ are nonzero.
Then 
$$e(f_1,\dots,f_c,sY^{a}+tf,h_1,\dots,h_n)=e(R[Y]/(sY^{a}+tf)_{(y_1,\dots,y_{n+c},Y)})$$
for sufficiently general elements  $(h_1,\dots,h_n)\in\mathfrak n\oplus\cdots\oplus\mathfrak n$
since  $(h_1,\dots,h_n)$ is a minimal reduction of  the ideal $(y_1,\dots,y_{n+c},Y)$ of $R[Y]/(sY^{a}+tf)_{(y_1,\dots,y_{n+c},Y)}$ (see for example Theorem 8.6.6 in \cite{HS}). 
Since $S_\mathfrak m\cong R[Y]/(sY^{a}+tf)_{(y_1,\dots,y_{n+c},Y)}$,
we have 
\begin{align*}
e(S_\mathfrak m)&=e(R[Y]/(sY^{a}+tf)_{(y_1,\dots,y_{n+c},Y)})\\
&=\sigma_A((f_1),\dots,(f_c),(Y^{a},f),\mathfrak n,\dots,\mathfrak n).
\end{align*}

\end{proof}

\begin{prop}\label{key proposition2}
Let $\mathbb D=(D,w)$ be an $n$-dimensional connected special datum and $J\in D$ be the maximal element of $D$ with $|J|\ge 2$.
Let $J_1,\dots,J_m$ be the elements of $D$ such that $J=J_1\cup\dots\cup J_m$ and
$J_i\sqsubset  J$ for $i=1,\dots,m$.
If $1=\mathrm{lct}(\mathfrak m_{{\mathbb D}})>\frac{\mathrm{lct}(\mathfrak m_{\mathbb D\setminus J})}{w(J_1)}$, then $$e(R_{\mathbb D})\ge\mathrm{lct}(\mathfrak m_{\mathbb D\setminus J})e(R_{\mathbb D\setminus J}).$$
\end{prop}

\begin{proof}
Let $a=w(J_1)$ and $f_1,\dots,f_{c},f$ be elements of the ideal $(y_1,\dots,y_{n+c})$ of the polynomial ring $\mathbb C[y_1,\dots,y_{n+c}]$ such that   
$$R_{{\mathbb D\setminus J}}\cong\mathbb C[y_1,\dots,y_{n+c}]/(f_1,\dots,f_{c})\ \ \ \mbox{and}$$
$$R_\mathbb D\cong\mathbb C[y_1,\dots,y_{n+c},Y]/(f_1,\dots,f_{c},Y^a-f).$$
Note that we can choose such  $f_1,\dots,f_{c},f$ by Lemma \ref{reduce lemma}.
Let $\mathbb D\setminus J=(D\setminus J,v)$ and $\boldsymbol{e}_i=(0,\cdots,0,\overset{i}{\check{1}},0,\cdots,0)\in\mathbb R^n$.
By Lemma \ref{log canonical threshold lemma}, there exist $J'_1,\dots,J'_k\in D\setminus J$ and $s,t_1,\dots,t_k,p,q\in\mathbb N$ such that 
$$(p,q)=1,$$
$$\frac{t_1}{s}+\cdots+\frac{t_k}{s}=\frac{\mathrm{lct}(\mathfrak m_{\mathbb D\setminus J})}{a}=\frac{p}{q}<1\ \ \mathrm{and}$$
$$\frac{at_1}{s}v(J'_1)\boldsymbol{e}_{J'_1}+\cdots+\frac{at_k}{s}v(J'_k)\boldsymbol{e}_{J'_k}=(1,\dots,1),$$
where $\boldsymbol{e}_{J'_i}=\sum_{j\in J'_i}\boldsymbol{e}_{j}$.
Let $(a,q)=d$, $a=a'd$ and $q=q'd$.
Then we have  $$sp=t_1q\cdots+t_kq\ \ \ \mathrm{and}$$
 $$({(x_1\cdots x_n)^{q'}})^{sp}=\Bigl( (x_{J'_1}^{a'pv(J'_1)})^{t_1}\cdots (x_{J'_k}^{a'pv(J'_k)})^{t_k}\Bigr)^{q}.$$
Let  $R=\mathbb C[x_{J'}^{a'pv(J')}|J'\in D\setminus J][x_J^{q'}]$ and $\mathfrak m_R=(x_{J'}^{a'pv(J')}|J'\in D\setminus J)+(x_J^{q'})\subset R$. 
Let  $S=\mathbb C[x_{J'}^{a'pv(J')}|J'\in D\setminus J]$ and $\mathfrak m_S=(x_{J'}^{a'pv(J')}|J'\in D\setminus J)\subset S$.
Note that 
$$S\cong R_{{\mathbb D\setminus J}}\  \mbox{and}$$ 
$$R\cong S[Y]/(Y^{a'p}-(x_{J_1}^{a'p}\cdots x_{J_m}^{a'p})^{q'})\cong R_{{\mathbb D\setminus J}}[Y]/(Y^{a'p}-(x_{J_1}\cdots x_{J_m})^{q'}).$$
Let $I$ be a minimal reduction of the maximal ideal  of the local ring $S_{\mathfrak m_S}$.
Since $({(x_1\cdots x_n)^{q'}})^{sp}=\Bigl( (x_{J'_1}^{a'pv(J'_1)})^{t_1}\cdots (x_{J'_k}^{a'pv(J'_k)})^{t_k}\Bigr)^{q}$, the integral closure of $IR_{\mathfrak m_R}$ contains $x_J^{q'}=(x_1\cdots x_n)^{q'}$.
This implies that  the integral closure of $IR_{\mathfrak m_R}$ is the maximal ideal of  $R_{\mathfrak m_R}$.
Hence  $IR_{\mathfrak m_R}$ is a minimal reduction of the maximal ideal of  $R_{\mathfrak m_R}$.
Thus \begin{align*}
e(R_{\mathfrak m_R}) & = \ell(R_{\mathfrak m_R}/IR_{\mathfrak m_R}) \\
 &=\ell\bigl{(}(S_{\mathfrak m_S}/I)[Y]/(Y^{a'p}-(x_{J_1}^{a'p}\cdots x_{J_m}^{a'p})^{q'})\bigr{)}\\
& =a'p\ell(S_{\mathfrak m_S}/I)\\
&=a'pe(R_{\mathbb D\setminus J})
\end{align*}
Let $A=\mathbb C[y_1,\dots,y_{n+c},Y]_{(y_1,\dots,y_{n+c},Y)}$ and $\mathfrak n$ be the maximal ideal of $A$.
By Lemma \ref{lemma2 of A2}, Lemma \ref{lemma2 of A2 integral closure} and Lemma \ref{sigma multiplicity lemma},
\begin{align*}
e(R_\mathbb D) & =\sigma_A((f_1),\dots,(f_c),(Y^a,f),\mathfrak n,\dots,\mathfrak n)\\
& =\frac{1}{q'}\sigma_A((f_1),\dots,(f_c),(Y^{a},f)^{q'},\mathfrak n,\dots,\mathfrak n)\\
& =\frac{1}{q'}\sigma_A((f_1),\dots,(f_c),(Y^{aq'},f^{q'}),\mathfrak n,\dots,\mathfrak n)\\
& \ge\frac{1}{q'}\sigma_A((f_1),\dots,(f_c),(Y^{a'p},f^{q'}),\mathfrak n,\dots,\mathfrak n)\\
&= \frac{1}{q'}e\bigl(\mathbb C[y_1,\dots,y_{n+c},Y]/(f_1,\dots,f_{c},Y^{a'p}-f^{q'})_{(y_1,\dots,y_{n+c},Y)}\bigr{)}\\
& = \frac{1}{q'}e(R_{\mathfrak m_R}) \\
&= \frac{a'p}{q'}e(R_{\mathbb D\setminus J}) \\
&=\mathrm{lct}(\mathfrak m_{\mathbb D\setminus J})e(R_{\mathbb D\setminus J}).
\end{align*}
The third equality holds since the integral closure of $(Y^{a},f)^{q'}$ is equal to the integral closure of $(Y^{aq'},f^{q'})$.
\end{proof}


\section{Upper bound of the multiplicity}

In this section, we give an upper bound of the multiplicity of abelian quotient complete intersection singularities.

\begin{defn}
Let $\mathbb D=(D,w)$ be a special datum and $J\in D$ with $|J|\ge2$.
We define an invariant $\delta(J)$ by 
$$\delta(J):=|\{J'\in D|J'\sqsubset  J\}|.$$
We define an invariant $m(D)$ by 
$$m(D):=\prod_{J\in D, |J|\ge2}\delta(J).$$
\end{defn}

\begin{thm}$\mathrm(${\rm See proof of Proposition 3.1  in \cite{W}}$\mathrm)$\label{watanabe multiplicity}
Let $\mathbb D=(D,w)$ be an $n$-dimensional special datum.
Then $$e(R_{\mathbb D})\le m(D)\le 2^{n-1}.$$
\end{thm}

\begin{cor}\label{emb 2n-1}
Let $\mathbb D=(D,w)$ be an $n$-dimensional special datum.
Then $$e(R_{\mathbb D})\le 2^{n-1}$$
and the equality holds if and only if $\mathrm{emb}(R_{\mathbb D})=2n-1$.
\end{cor}

\begin{proof}
The inequality follows from Theorem \ref{watanabe multiplicity}.

If $\mathrm{emb}(R_{\mathbb D})=2n-1$, then $e(R_{\mathbb D})\ge2^{n-1}$ (See Example 12.4.9 in \cite{F}).
By Theorem \ref{watanabe multiplicity}, we have $e(R_{\mathbb D})=2^{n-1}$.

We assume that $e(R_{\mathbb D})=2^{n-1}$.
By Theorem \ref{watanabe multiplicity}, we have $m(D)=2^{n-1}$.
Note that $$\sum_{J\in D, |J|\ge2}(\delta(J)-1)=n-1$$
(See proof of Proposition 3.1  in \cite{W}).
Therefore $$2^{n-1}=m(D)=\prod_{J\in D, |J|\ge2}\delta(J)\le 2^{\sum_{J\in D, |J|\ge2}(\delta(J)-1)}=2^{n-1}.$$
This implies that  $|\{J\in D||J|\ge2\}|=n-1$ and $\delta(J)=2$ for $J\in D$ with $|J|\ge2$.
By Proposition \ref{embedding dimension}, we have $\mathrm{emb}(R_{\mathbb D})=2n-1$.
\end{proof}

In order to prove Theorem \ref{Main theorem}, we need the following inequality.
\begin{lem}\label{lemma inequality}
If $a\in\mathbb N$, $b\in \mathbb R$ and  $2\le a\le b$,  then $a\le 2^{\lceil b\rceil-\lceil\frac{b}{a} \rceil}$ and the equality holds if and only if  $a=2$ and $\lceil b\rceil-\lceil\frac{b}{a} \rceil=1$.
\end{lem}

\begin{proof}
Since $a-1\le\lfloor\frac{b}{a} \rfloor(a-1)\le  \lceil b\rceil-\lceil\frac{b}{a} \rceil$,
we have $a\le2^{a-1}\le 2^{\lceil b\rceil-\lceil\frac{b}{a} \rceil}$.
By the above inequalities, $a=2^{\lceil b\rceil-\lceil\frac{b}{a} \rceil}$ if and only if  $a=2$ and $\lceil b\rceil-\lceil\frac{b}{a} \rceil=1$.
\end{proof}

\begin{thm}\label{Main theorem}
Let $\mathbb D=(D,w)$ be an $n$-dimensional  special datum.
Then $$e(R_{\mathbb D})\le 2^{n-\lceil\mathrm{lct}(\mathfrak m_{{\mathbb D}})\rceil}$$
and the equality holds if and only if $\mathrm{emb}(R_{\mathbb D})=2n-\lceil\mathrm{lct}(\mathfrak m_{{\mathbb D}})\rceil$.
\end{thm}

\begin{proof}
We prove this by induction on the dimension of $\mathbb D=(D,w)$.
If the dimension of $R_{\mathbb D}$ is $1$, then  $e(R_{\mathbb D})=1$ and $\mathrm{lct}(\mathfrak m_{\mathbb D})=1$.
Therefore the theorem holds when $\mathrm{dim}R_{\mathbb D}=1$.
Now suppose that  $n\ge 2$ and the theorem holds for any special datum  of dimension at most $n-1$.

We assume that $\mathbb D$ is not connected.
Note that   the dimension of a connected component of $\mathbb D$ is less than $n$.
Let $J_1,\dots,J_m$ be the maximal elements of $D$ and $n_i=\mathrm{dim}R_{\mathbb D_{J_i}}$.
Then we have $e(R_\mathbb D) =e(R_{\mathbb D_{J_1}})\cdots e(R_{\mathbb D_{J_m}})$,
  $\mathrm{emb}(R_{\mathbb D})=\mathrm{emb}(R_{\mathbb D_{J_1}})+\cdots+\mathrm{emb}(R_{\mathbb D_{J_m}})$ and
 $\mathrm{lct}(\mathfrak m_{{\mathbb D}})=\mathrm{lct}(\mathfrak m_{\mathbb D_{J_1}})+\cdots+\mathrm{lct}(\mathfrak m_{\mathbb D_{J_m}})$ by Proposition \ref{non-connected} and Proposition \ref{log canonical threshold}.
By hypothesis, we have for $i=1,\dots,m$
$$e(R_{\mathbb D_{J_i}})\le 2^{n_i-\lceil\mathrm{lct}(\mathfrak m_{{\mathbb D_{J_i}}})\rceil}$$
and the equality holds if and only if $\mathrm{emb}(R_{\mathbb D_{J_i}})=2n_i-\lceil\mathrm{lct}(\mathfrak m_{{\mathbb D_{J_i}}})\rceil$.
Therefore
\begin{align*}
e(R_\mathbb D) & =e(R_{\mathbb D_{J_1}})\cdots e(R_{\mathbb D_{J_m}})\\
&\le 2^{n_1-\lceil\mathrm{lct}(\mathfrak m_{{\mathbb D_{J_1}}})\rceil}\cdots2^{n_m-\lceil\mathrm{lct}(\mathfrak m_{{\mathbb D_{J_m}}})\rceil}\\
& \le 2^{n-\lceil\mathrm{lct}(\mathfrak m_{{\mathbb D}})\rceil}.
\end{align*}
Note that 
$$\mathrm{emb}(R_{\mathbb D})\le 2n-\lceil\mathrm{lct}(\mathfrak m_{{\mathbb D_{J_1}}})\rceil\cdots-\lceil\mathrm{lct}(\mathfrak m_{{\mathbb D_{J_m}}})\rceil\le 2n-\lceil\mathrm{lct}(\mathfrak m_{{\mathbb D}})\rceil$$
 by Proposition \ref{embedding dimension log canonical threshold1}.
The equality $e(R_{\mathbb D})=2^{n-\lceil\mathrm{lct}(\mathfrak m_{{\mathbb D}})\rceil}$ holds if and only if
$\mathrm{emb}(R_{\mathbb D_{J_i}})=2n_i-\lceil\mathrm{lct}(\mathfrak m_{{\mathbb D_{J_i}}})\rceil$ for $i=1,\dots,m$
and $\lceil\mathrm{lct}(\mathfrak m_{{\mathbb D_{J_1}}})\rceil+\dots+\lceil\mathrm{lct}(\mathfrak m_{{\mathbb D_{J_m}}})\rceil=\lceil\mathrm{lct}(\mathfrak m_{{\mathbb D}})\rceil$
if and only if
$\mathrm{emb}(R_{\mathbb D})=2n-\lceil\mathrm{lct}(\mathfrak m_{{\mathbb D}})\rceil$. 
Therefore the theorem holds if $\mathbb D$ is not connected.

We assume that $\mathbb D=(D,w)$ is connected.
Let $J$ be the maximal element of $D$ and $J_1,\dots,J_m$ be the elements of $D$ such that $J=J_1\cup\dots\cup J_m$ and
$J_i\sqsubset  J$ for $i=1,\dots,m$.
Let $a=w(J_1)$  and $n_i=\mathrm{dim}R_{\mathbb D_{J_i}}$.
Note that $w(J_1)=\cdots=w(J_m)$ by the definition of a special datum and  the dimension of a connected component of $\mathbb D\setminus J$ is less than $n$.

If $\mathrm{lct}(\mathfrak m_{{\mathbb D}})=\frac{\mathrm{lct}(\mathfrak m_{\mathbb D\setminus J})}{a}$, then by Proposition \ref{lct=lct} $$e(R_{\mathbb D})=ae(R_{\mathbb D\setminus J}).$$
Therefore by hypothesis, Proposition \ref{non-connected}  and Lemma \ref{lemma inequality}, we have 
\begin{align*}
e(R_\mathbb D) & =ae(R_{\mathbb D\setminus J})\\
&=ae(R_{\mathbb D_{J_1}})\cdots e(R_{\mathbb D_{J_m}})\\
&\le a2^{n_1-\lceil\mathrm{lct}(\mathfrak m_{\mathbb D_{J_1}})\rceil}\cdots2^{n_m-\lceil\mathrm{lct}(\mathfrak m_{{\mathbb D_{J_m}}})\rceil}\\
&\le a2^{n-\lceil\mathrm{lct}(\mathfrak m_{\mathbb D\setminus J})\rceil}\\
&\le 2^{n-\lceil\mathrm{lct}(\mathfrak m_{\mathbb D})\rceil}
\end{align*}
By Proposition \ref{embedding dimension log canonical threshold1}, Lemma \ref{embedding dimension log canonical threshold2} and Lemma \ref{lemma inequality}, the equality $e(R_{\mathbb D})=2^{n-\lceil\mathrm{lct}(\mathfrak m_{{\mathbb D}})\rceil}$ holds if and only if
$$\mathrm{emb}(R_{\mathbb D_{J_i}})=2n_i-\lceil\mathrm{lct}(\mathfrak m_{{\mathbb D_{J_i}}})\rceil\ \ \mbox{for}\ i=1,\dots,m,$$
 $$\lceil\mathrm{lct}(\mathfrak m_{{\mathbb D_{J_1}}})\rceil+\dots+\lceil\mathrm{lct}(\mathfrak m_{{\mathbb D_{J_m}}})\rceil=\lceil\mathrm{lct}(\mathfrak m_{{\mathbb D\setminus J}})\rceil,$$
 $$a=2\ \  \mbox{and}\ \  \lceil\mathrm{lct}(\mathfrak m_{{\mathbb D\setminus J}})\rceil-\lceil\mathrm{lct}(\mathfrak m_{\mathbb D})\rceil=1$$
if and only if
$\mathrm{emb}(R_{\mathbb D})=2n-\lceil\mathrm{lct}(\mathfrak m_{{\mathbb D}})\rceil$. 
Therefore the theorem holds if $\mathbb D$ is  connected and $\mathrm{lct}(\mathfrak m_{{\mathbb D}})=\frac{\mathrm{lct}(\mathfrak m_{\mathbb D\setminus J})}{a}$.

If $1=\mathrm{lct}(\mathfrak m_{{\mathbb D}})>\frac{\mathrm{lct}(\mathfrak m_{\mathbb D\setminus J})}{a}$,
then by Corollary \ref{emb 2n-1}, we have $e(R_{\mathbb D})\le 2^{n-1}$
and the equality holds if and only if $\mathrm{emb}(R_{\mathbb D})=2n-1$.
Therefore the theorem holds if $\mathbb D$ is  connected and $\mathrm{lct}(\mathfrak m_{{\mathbb D}})=1$.
\end{proof}

Theorem \ref{Main Theorem A} follows from Theorem \ref{Watanabe main theorem} and Theorem \ref{Main theorem}.

\section{Lower bound of the multiplicity}

In this section, we give a lower bound of the multiplicity of abelian quotient complete intersection singularities.

\begin{defn}
Let $\mathbb D=(D,w)$ be an $n$-dimensional connected special datum and $J\in D$ be the maximal element of $D$.
Then we define 
\[
  \alpha(\mathbb D) = \begin{cases}
    \mathrm{min}\{\mathrm{lct}(\mathfrak m_{{{\mathbb D}}\setminus J}),{w(J')} \}& ( \mathrm{if}\ n\ge 2\ \mathrm{and}\ J'\sqsubset  J) \\
    1 & (\mathrm{if}\ n=1)
  \end{cases}
\]
and 
\[
  \beta(\mathbb D) = \begin{cases}
    {w(J')} & ( \mathrm{if}\ n\ge 2\ \mathrm{and}\ J'\sqsubset  J) \\
    1 & (\mathrm{if}\ n=1)
  \end{cases}
.\]

\end{defn}

\begin{prop}
Let $\mathbb D=(D,w)$ be an $n$-dimensional special datum.
If $R_{\mathbb D}$ is a hypersurface {\rm(}i.e. $\mathrm{emb}(R_{\mathbb D})=\mathrm{dim}R_{\mathbb D}+1${\rm)}, then
$$e(R_\mathbb D)=\prod_{J\in D}\alpha(\mathbb D_J).$$
\end{prop}

\begin{proof}
Let $J_1,\dots,J_m$ be the maximal elements of $D$.
By Proposition \ref{embedding dimension} and Proposition \ref{non-connected}, we may assume that 
$\mathrm{emb}(R_{\mathbb D_{J_1}})=\mathrm{dim}R_{\mathbb D_{J_1}}+1$ and $\mathrm{dim}R_{\mathbb D_{J_i}}=1$ for $i\ge 2$.
Let $J'$ be a element of $D$ with $J'\sqsubset  J_1$, $n_1=\mathrm{dim}R_{\mathbb D_{J_1}}$ and $a=w(J')$.
Then $$R_{\mathbb D_{J_1}}\cong \mathbb C[x_1,\dots,x_{n_1+1}]/(x_{n_1+1}^a-x_1\cdots x_{n_1})$$ 
by Lemma \ref{reduce lemma}.
Therefore we have 
$e(R_{\mathbb D_1})=\mathrm{min}\{a,n_1\}$.
Thus 
$e(R_{\mathbb D})=\mathrm{min}\{a,n_1\}$ by Proposition \ref{non-connected}.

On the other hand, by Proposition \ref{log canonical threshold}, we have $\mathrm{lct}(\mathfrak m_{{{\mathbb D_{J_1}}}\setminus J1})=n_1$.
Note that $|J|=1$ for any element $J$ of $D$ with $J\neq J_1$.
Hence $$\prod_{J\in D}\alpha(\mathbb D_J)=\alpha(\mathbb D_{J_1})=\mathrm{min}\{a,n_1\}.$$
Therefore $e(R_\mathbb D)=\prod_{J\in D}\alpha(\mathbb D_J).$
\end{proof}

\begin{prop}\label{Key Prop 1}
Let $\mathbb D=(D,w)$ be an $n$-dimensional special datum.
Then
$$e(R_\mathbb D)\ge\prod_{J\in D}\alpha(\mathbb D_J)$$
and the equality $e(R_\mathbb D)=\prod_{J\in D}\alpha(\mathbb D_J)$ holds  if $\alpha(\mathbb D_J)=\beta(\mathbb D_J)$ for every $J\in  D$.
\end{prop}

\begin{proof}
We will prove this by induction on the dimension of $\mathbb D=(D,w)$.
If the dimension of $R_{\mathbb D}$ is $1$, then $e(R_{\mathbb D})=\alpha(\mathbb D)=\beta(\mathbb D)=1$.
Therefore the proposition holds when $\mathrm{dim}R_{\mathbb D}=1$.
Now suppose that  $n\ge 2$ and the proposition holds for any special datum  of dimension at most $n-1$.

We assume that $\mathbb D$ is not connected.
Let $J_1,\dots,J_m$ be the maximal elements of $D$ and $n_i=\mathrm{dim}R_{\mathbb D_{J_i}}$.
Then we have $e(R_\mathbb D) =e(R_{\mathbb D_{J_1}})\cdots e(R_{\mathbb D_{J_m}})$ by Proposition \ref{non-connected}.
By hypothesis, we have
$$e(R_\mathbb D)  =e(R_{\mathbb D_{J_1}})\cdots e(R_{\mathbb D_{J_m}})
\ge \prod_{J\in D_{J_1}}\alpha(\mathbb D_{J})\cdots\prod_{J\in D_{J_m}}\alpha(\mathbb D_{J})
= \prod_{J\in D}\alpha(\mathbb D_J).
$$
Therefore the inequality holds if  $\mathbb D$ is not connected.
We assume that  $\alpha(\mathbb D_J)=\beta(\mathbb D_J)$ for every $J\in  D$.
Then by hypothesis, we have 
$$
e(R_\mathbb D)  =e(R_{\mathbb D_{J_1}})\cdots e(R_{\mathbb D_{J_m}})
= \prod_{J\in D_{J_1}}\alpha(\mathbb D_{J})\cdots\prod_{J\in D_{J_m}}\alpha(\mathbb D_{J})
= \prod_{J\in D}\alpha(\mathbb D_J).
$$
Therefore the proposition holds if  $\mathbb D$ is not connected.

We assume that $\mathbb D$ is  connected.
Let $J$ be the maximal element of $D$ and $J_1,\dots,J_m$ be the elements of $D$ such that $J=J_1\cup\dots\cup J_m$ and
$J_i\sqsubset  J$ for $i=1,\dots,m$.
Then we have $e(R_{\mathbb D\setminus J}) =e(R_{\mathbb D_{J_1}})\cdots e(R_{\mathbb D_{J_m}})$ by Proposition \ref{non-connected}.
Then by Proposition \ref{lct=lct} and Proposition \ref{key proposition2},
we have 
$$e(R_\mathbb D)\ge\alpha(\mathbb D)\prod_{1\le i\le m}e(R_{\mathbb D_{J_1}})\cdots e(R_{\mathbb D_{J_m}}).$$
Note that $\mathbb D=\mathbb D_J$.
By hypothesis, we have
\begin{align*}
e(R_\mathbb D) & \ge\alpha(\mathbb D)\prod_{1\le i\le m}e(R_{\mathbb D_{J_1}})\cdots e(R_{\mathbb D_{J_m}})\\
&\ge \alpha(\mathbb D)\prod_{K\in D_{J_1}}\alpha(\mathbb D_{K})\cdots\prod_{K\in D_{J_m}}\alpha(\mathbb D_{K})
= \prod_{K\in D}\alpha(\mathbb D_{K}).
\end{align*}
Therefore the inequality holds if  $\mathbb D$ is  connected.
We assume that  $\alpha(\mathbb D_K)=\beta(\mathbb D_K)$ for every $K\in  D$.
Then since $\mathrm{lct}(\mathfrak m_{{{\mathbb D}}\setminus J})\ge{w(J_1)}$,
we have $\mathrm{lct}(\mathfrak m_{{\mathbb D}})=\frac{\mathrm{lct}(\mathfrak m_{\mathbb D\setminus J})}{w(J_1)}$
by Proposition \ref{log canonical threshold max 1 lct}.
By  hypothesis and Proposition \ref{lct=lct}, 
\begin{align*}
e(R_\mathbb D) & =\alpha(\mathbb D)\prod_{1\le i\le m}e(R_{\mathbb D_{J_1}})\cdots e(R_{\mathbb D_{J_m}})\\
&= \alpha(\mathbb D)\prod_{K\in D_{J_1}}\alpha(\mathbb D_{K})\cdots\prod_{K\in D_{J_m}}\alpha(\mathbb D_{K})
= \prod_{K\in D}\alpha(\mathbb D_{K}).
\end{align*}
Therefore the proposition holds if  $\mathbb D$ is  connected.
\end{proof}

In order to prove Proposition \ref{Key Prop 2}, we need the following inequality.
\begin{lem}\label{easy inequality}
For positive real numbers $x_1,\dots,x_n,c_1,\dots,c_n\in\mathbb R_{>0}$,
$$\Bigl(\frac{x_1}{c_1}\Bigr)^{x_1}\cdots\Bigl(\frac{x_n}{c_n}\Bigr)^{x_n}\ge\Bigl(\frac{x_1+\cdots+x_n}{c_1+\cdots+c_n}\Bigr)^{x_1+\cdots+x_n}.$$
Moreover, the equality holds if and only if $\frac{x_1}{c_1}=\cdots=\frac{x_n}{c_n}$.
\end{lem}

\begin{proof}
Since the logarithm function $f(x)=\mathrm{log}{x}$ is concave on its domain $(0,\infty )$, we have 
$$\frac{x_1}{x_1+\cdots+x_n}\mathrm{log}\frac{c_1}{x_1}+\cdots+\frac{x_n}{x_1+\cdots+x_n}\mathrm{log}\frac{c_n}{x_n}\le \mathrm{log}\frac{c_1+\cdots+c_n}{x_1+\cdots+x_n}$$
and the equality holds if and only if $\frac{x_1}{c_1}=\cdots=\frac{x_n}{c_n}$.
This implies that 
$$\Bigl(\frac{c_1}{x_1}\Bigr)^{x_1}\cdots\Bigl(\frac{c_n}{x_n}\Bigr)^{x_n}\le\Bigl(\frac{c_1+\cdots+c_n}{x_1+\cdots+x_n}\Bigr)^{x_1+\cdots+x_n}$$
and the equality holds if and only if $\frac{x_1}{c_1}=\cdots=\frac{x_n}{c_n}$.
\end{proof}

\begin{prop}\label{Key Prop 2}
Let $\mathbb D=(D,w)$ be an $n$-dimensional special datum.
Let $\mathfrak a_\mathbb D\subset \mathbb C[x_1,\dots,x_n]$ be the ideal generated by $x_J^{w(J)}$ for $J\in D$.
Then
$$\prod_{J\in D}\alpha(\mathbb D_J)\ge \frac{1}{|G_\mathbb D|}   \biggl(\frac{n}{\mathrm{lct}(\mathfrak m_\mathbb D)}  \biggr)^n$$
and the equality $\prod_{J\in D}\alpha(\mathbb D_J)=\frac{1}{|G_\mathbb D|}   \bigl(\frac{n}{\mathrm{lct}(\mathfrak m_\mathbb D)}  \bigr)^n$ holds if and only if there is a positive integer $q$ such that the integral closure 
$\overline{\mathfrak a_\mathbb D}$ of $\mathfrak a_\mathbb D$ is equal to $(x_1,\dots,x_n)^q$.
Moreover, in this case $$q=\frac{n}{\mathrm{lct}(\mathfrak m_\mathbb D)}=w(\{1\})=\cdots=w(\{n\})$$
and $\alpha(\mathbb D_J)=\beta(\mathbb D_J)$ for every $J\in  D$.
\end{prop}

\begin{proof}
We will prove this by induction on the dimension of $\mathbb D=(D,w)$.
If the dimension of $R_{\mathbb D}$ is $1$, then $|G_{\mathbb D}|=1$, $\mathrm{lct}(\mathfrak m_{\mathbb D})=1$, $ \alpha(\mathbb D)=\beta(\mathbb D)=1$, $w(\{1\})=1$ and $\mathfrak a_\mathbb D=x_1$.
Therefore the proposition holds when $\mathrm{dim}R_{\mathbb D}=1$.
Now suppose that  $n\ge 2$ and the proposition holds for any special datum  of dimension at most $n-1$.

We assume that $\mathbb D$ is not connected.
Let $J_1,\dots,J_m$ be the maximal elements of $D$ and $n_i=\mathrm{dim}R_{\mathbb D_{J_i}}$.
Then we have $e(R_\mathbb D) =e(R_{\mathbb D_{J_1}})\cdots e(R_{\mathbb D_{J_m}})$,
$|G_{\mathbb D}|=|G_{\mathbb D_{J_1}}|\cdots|G_{\mathbb D_{J_m}}|$
   and
 $\mathrm{lct}(\mathfrak m_{{\mathbb D}})=\mathrm{lct}(\mathfrak m_{\mathbb D_{J_1}})+\cdots+\mathrm{lct}(\mathfrak m_{\mathbb D_{J_m}})$ by Proposition \ref{non-connected} and Proposition \ref{log canonical threshold}. 
By hypothesis and Lemma \ref{easy inequality}, we have
\begin{align*}
\prod_{J\in D}\alpha(\mathbb D_J)
 &=\prod_{J\in D_{J_1}}\alpha(\mathbb D_{J})\cdots\prod_{J\in D_{J_m}}\alpha(\mathbb D_{J})\\
&\ge \frac{1}{|G_{\mathbb D_{J_1}}|}\biggl(\frac{n_1}{\mathrm{lct}(\mathfrak m_{\mathbb D_{J_1}})}  \biggr)^{n_1}\cdots\frac{1}{|G_{\mathbb D_{J_m}}|}   \biggl(\frac{n_m}{\mathrm{lct}(\mathfrak m_{\mathbb D_{J_m}})}  \biggr)^{n_m}\\
&\ge \frac{1}{|G_\mathbb D|}\biggl(\frac{n}{\mathrm{lct}(\mathfrak m_\mathbb D)}\biggr)^n.
\end{align*}
Therefore the inequality holds if $\mathbb D$ is not connected.

Let  $n_0=0$ and $N_{i}=n_0+\cdots+n_{i}$ for $i=0,\dots,m-1$. 
Then we may assume that $J_i=\{N_{i-1}+1,\dots,N_{i-1}+n_i\}$ for $i=1,\dots,m$.
Let $\mathfrak a_{\mathbb D_{J_i}}\subset \mathbb C[x_{N_{i-1}+1},\dots,x_{N_{i-1}+n_i}]$ be the ideal generated by $x_J^{w(J)}$ for $J\in D_{J_i}$.

We assume that  $$\prod_{J\in D}\alpha(\mathbb D_J)=\frac{1}{|G_\mathbb D|}   \biggl(\frac{n}{\mathrm{lct}(\mathfrak m_\mathbb D)}  \biggr)^n.$$
Then we have for $i=1,\dots,m$
$$\prod_{J\in D_{J_i}}\alpha(\mathbb D_{J})=\frac{1}{|G_{\mathbb D_{J_i}}|}\biggl(\frac{n_i}{\mathrm{lct}(\mathfrak m_{\mathbb D_{J_i}})}  \biggr)^{n_i},\  \ \frac{n_1}{\mathrm{lct}(\mathfrak m_{\mathbb D_{J_1}})} =\cdots=\frac{n_m}{\mathrm{lct}(\mathfrak m_{\mathbb D_{J_m}})} $$ since  the above two inequalities are equalities.
By hypothesis, we have  for $i=1,\dots,m$, $\overline{\mathfrak a_{\mathbb D_i}}=(x_{N_{i-1}+1},\dots,x_{N_{i-1}+n_i})^{w(\{N_{i-1}+1\})}, $
$$\frac{n_i}{\mathrm{lct}(\mathfrak m_{\mathbb D_{J_i}})}=w(\{{N_{i-1}+1}\})=\cdots=w(\{N_{i-1}+n_i\})$$
and  $\alpha(\mathbb D_J)=\beta(\mathbb D_J)$ for every $J\in  D_{J_i}$.
Therefore we have $\overline{\mathfrak a_{\mathbb D}}=(x_{1},\dots,x_{n})^{w(\{1\})}$, 
$$\frac{n}{\mathrm{lct}(\mathfrak m_\mathbb D)}=w(\{1\})=\cdots=w(\{n\})$$ and $\alpha(\mathbb D_J)=\beta(\mathbb D_J)$ for every $J\in  D$.

We assume that there is a positive integer $q$ such that 
$\overline{\mathfrak a_\mathbb D}=(x_1,\dots,x_n)^q$.
Then we have $\overline{\mathfrak a_{\mathbb D_{J_i}}}=(x_{N_{i-1}+1},\dots,x_{N_{i-1}+n_i})^q$.
By hypothesis, we have  $$\prod_{J\in D_{J_i}}\alpha(\mathbb D_{J})=\frac{1}{|G_{\mathbb D_{J_i}}|}\biggl(\frac{n_i}{\mathrm{lct}(\mathfrak m_{\mathbb D_{J_i}})}  \biggr)^{n_i}\ \ \mbox{and}$$
$$ q=\frac{n_i}{\mathrm{lct}(\mathfrak m_{\mathbb D_{J_i}})}=w(\{{N_{i-1}+1}\})=\cdots=w(\{N_{i-1}+n_i\})$$ for $i=1,\dots,m$.
By Lemma \ref{easy inequality}, we have $$\prod_{J\in D}\alpha(\mathbb D_J)=\frac{1}{|G_\mathbb D|}   \biggl(\frac{n}{\mathrm{lct}(\mathfrak m_\mathbb D)}  \biggr)^n.$$
Therefore the proposition holds if $\mathbb D$ is not connected.

We assume that $\mathbb D$ is  connected.
Let $J$ be the maximal element of $D$, $J_1,\dots,J_m$ be the elements of $D$ such that $J=J_1\cup\dots\cup J_m$ and
$J_i\sqsubset  J$ for $i=1,\dots,m$.
Let $a=w(J_1)$.

If $\mathrm{lct}(\mathfrak m_{{\mathbb D}})=\frac{\mathrm{lct}(\mathfrak m_{\mathbb D\setminus J})}{a}$, then $\alpha(\mathbb D)=a.$
By hypothesis, Lemma \ref{group reduction}, Proposition \ref{log canonical threshold} and Lemma \ref{easy inequality}, we have
\begin{align*}
\prod_{K\in D}\alpha(\mathbb D_K)
 &=a\prod_{K\in D_{J_1}}\alpha(\mathbb D_{K})\cdots\prod_{K\in D_{J_m}}\alpha(\mathbb D_{K})\\
&\ge a\frac{1}{|G_{\mathbb D_{J_1}}|}\biggl(\frac{n_1}{\mathrm{lct}(\mathfrak m_{\mathbb D_{J_1}})}  \biggr)^{n_1}\cdots\frac{1}{|G_{\mathbb D_{J_m}}|}   \biggl(\frac{n_m}{\mathrm{lct}(\mathfrak m_{\mathbb D_{J_m}})}  \biggr)^{n_m}\\
&\ge \frac{a}{|G_{\mathbb D\setminus J}|}\biggl(\frac{n}{\mathrm{lct}(\mathfrak m_{\mathbb D\setminus J})}\biggr)^n\\
&= \frac{1}{|G_\mathbb D|}\biggl(\frac{n}{\mathrm{lct}(\mathfrak m_\mathbb D)}\biggr)^n.
\end{align*}

If $1=\mathrm{lct}(\mathfrak m_{{\mathbb D}})>\frac{\mathrm{lct}(\mathfrak m_{\mathbb D\setminus J})}{a}$, then $\alpha(\mathbb D)=\mathrm{lct}(\mathfrak m_{\mathbb D\setminus J}).$
By hypothesis, Lemma \ref{group reduction}, Proposition \ref{log canonical threshold} and Lemma \ref{easy inequality}, we have
\begin{align*}
\prod_{K\in D}\alpha(\mathbb D_J)
 &=\mathrm{lct}(\mathfrak m_{\mathbb D\setminus J})\prod_{K\in D_{J_1}}\alpha(\mathbb D_{K})\cdots\prod_{K\in D_{J_m}}\alpha(\mathbb D_{K})\\
&\ge \mathrm{lct}(\mathfrak m_{\mathbb D\setminus J})\frac{1}{|G_{\mathbb D_{J_1}}|}\biggl(\frac{n_1}{\mathrm{lct}(\mathfrak m_{\mathbb D_{J_1}})}  \biggr)^{n_1}\cdots\frac{1}{|G_{\mathbb D_{J_m}}|}   \biggl(\frac{n_m}{\mathrm{lct}(\mathfrak m_{\mathbb D_{J_m}})}  \biggr)^{n_m}\\
&\ge \frac{\mathrm{lct}(\mathfrak m_{\mathbb D\setminus J})}{|G_{\mathbb D\setminus J}|}\biggl(\frac{n}{\mathrm{lct}(\mathfrak m_{\mathbb D\setminus J})}\biggr)^n\\
&> \frac{1}{|G_\mathbb D|}\biggl(\frac{n}{\mathrm{lct}(\mathfrak m_\mathbb D)}\biggr)^n.
\end{align*}
Therefore the inequality holds if $\mathbb D$ is  connected.

Let  $n_0=0$,   $N_{i}=n_0+\cdots+n_{i}$ for $i=0,\dots,m-1$ and $\mathbb D\setminus J=(D\setminus J,v)$.
Then we may assume that $J_i=\{N_{i-1}+1,\dots,N_{i-1}+n_i\}$ for $i=1,\dots,m$.
Let $\mathfrak a_{\mathbb D_{J_i}}\subset \mathbb C[x_{N_{i-1}+1},\dots,x_{N_{i-1}+n_i}]$ be the ideal generated by $x_K^{v(K)}$ for $K\in D_{J_i}$. 
Note that $v(K)=\frac{w(K)}{a}$.

We assume that  $$\prod_{K\in D}\alpha(\mathbb D_K)=\frac{1}{|G_\mathbb D|}   \biggl(\frac{n}{\mathrm{lct}(\mathfrak m_\mathbb D)}  \biggr)^n.$$
Then we have for $i=1,\dots,m$
$$\mathrm{lct}(\mathfrak m_{{\mathbb D}})=\frac{\mathrm{lct}(\mathfrak m_{\mathbb D\setminus J})}{a}, \prod_{K\in D_{J_i}}\alpha(\mathbb D_{K})=\frac{1}{|G_{\mathbb D_{J_i}}|}\biggl(\frac{n_i}{\mathrm{lct}(\mathfrak m_{\mathbb D_{J_i}})}  \biggr)^{n_i},$$   
$$\frac{n_1}{\mathrm{lct}(\mathfrak m_{\mathbb D_{J_1}})} =\cdots=\frac{n_m}{\mathrm{lct}(\mathfrak m_{\mathbb D_{J_m}})} $$ by the  above discussion.
By hypothesis, we have  for $i=1,\dots,m$, $$\overline{\mathfrak a_{\mathbb D_i}}=(x_{N_{i-1}+1},\dots,x_{N_{i-1}+n_i})^{w(\{N_{i-1}+1\})/a},$$ 
$$\frac{n_i}{\mathrm{lct}(\mathfrak m_{\mathbb D_{J_i}})}=\frac{w(\{{N_{i-1}+1}\})}{a}=\cdots=\frac{w(\{N_{i-1}+n_i\})}{a}$$ and  $\alpha(\mathbb D_K)=\beta(\mathbb D_K)$ for every $K\in  D_{J_i}$.
Therefore we have $\overline{\mathfrak a_{\mathbb D}}=(x_{1},\dots,x_{n})^{w(\{1\})}$, 
$$\frac{n}{\mathrm{lct}(\mathfrak m_\mathbb D)}=w(\{1\})=\cdots=w(\{n\})$$
 and $\alpha(\mathbb D_K)=\beta(\mathbb D_K)$ for every $K\in  D$.

We assume that there is a positive integer $q$ such that 
$\overline{\mathfrak a_\mathbb D}=(x_1,\dots,x_n)^q$.
Since ${\mathfrak a_\mathbb D}$ is a monomial ideal, 
$$\overline{\mathfrak a_\mathbb D}=\{x_1^{m_1}\cdots x_n^{m_n}|\ (m_1,\dots,m_n)\in \mathrm{Newt}({\mathfrak a_\mathbb D})\cap \mathbb Z^n\}$$ (see for example Proposition 1.4.6 in \cite{HS}).
Therefore we have $q=w(\{1\})=\cdots=w(\{n\})$.
We have $\overline{\mathfrak a_{\mathbb D_{J_i}}}=(x_{N_{i-1}+1},\dots,x_{N_{i-1}+n_i})^{q/a}$ since $\overline{\mathfrak a_\mathbb D}=(x_1,\dots,x_n)^q$.
By Lemma \ref{log canonical threshold lemma}, ${\mathrm{lct}(\mathfrak m_{\mathbb D})}=\frac{n}{q}$ and ${\mathrm{lct}(\mathfrak m_{\mathbb D_{J_i}})}=\frac{n_ia}{q}$.
These imply that 
$$\mathrm{lct}(\mathfrak m_{{\mathbb D}})=\frac{\mathrm{lct}(\mathfrak m_{\mathbb D\setminus J})}{a}\ \ \ \mbox{and}\ \ \frac{n_1}{\mathrm{lct}(\mathfrak m_{\mathbb D_{J_1}})} =\cdots=\frac{n_m}{\mathrm{lct}(\mathfrak m_{\mathbb D_{J_m}})}. $$ 
Therefore we have $\alpha(\mathbb D)=a$ as  $\mathrm{lct}(\mathfrak m_{{\mathbb D}})\ge 1$.
By hypothesis, we have  $$\prod_{K\in D_{J_i}}\alpha(\mathbb D_{K})=\frac{1}{|G_{\mathbb D_{J_i}}|}\biggl(\frac{n_i}{\mathrm{lct}(\mathfrak m_{\mathbb D_{J_i}})}  \biggr)^{n_i}$$  for $i=1,\dots,m$.
By Lemma \ref{group reduction} and Lemma \ref{easy inequality}, we have $$\prod_{K\in D}\alpha(\mathbb D_K)=\frac{1}{|G_\mathbb D|}   \biggl(\frac{n}{\mathrm{lct}(\mathfrak m_\mathbb D)}  \biggr)^n.$$
Therefore the proposition holds if $\mathbb D$ is  connected.

\end{proof}


\begin{thm}
Let $\mathbb D=(D,w)$ be an $n$-dimensional special datum.
Let $\mathfrak a_\mathbb D\subset \mathbb C[x_1,\dots,x_n]$ be the ideal generated by $x_J^{w(J)}$ for $J\in D$. Then
$$e(R_\mathbb D)\ge\prod_{J\in D}\alpha(\mathbb D_J)\ge \frac{1}{|G_\mathbb D|}   \biggl(\frac{n}{\mathrm{lct}(\mathfrak m_\mathbb D)}  \biggr)^n.$$
Moreover, 
\begin{enumerate}
\item[\rm{(1)}]
the equality $e(R_\mathbb D)=\prod_{J\in D}\alpha(\mathbb D_J)$ holds  if $\alpha(\mathbb D_J)=\beta(\mathbb D_J)$ for every $J\in  D$.
\item[\rm{(2)}]
The equality $\prod_{J\in D}\alpha(\mathbb D_J)=\frac{1}{|G_\mathbb D|}   \bigl(\frac{n}{\mathrm{lct}(\mathfrak m_\mathbb D)}  \bigr)^n$ holds if and only if there is a positive integer $q$ such that 
the integral closure 
$\overline{\mathfrak a_\mathbb D}$ of $\mathfrak a_\mathbb D$ is equal to $(x_1,\dots,x_n)^q$.
 Furthermore, in this case $$e(R_\mathbb D)=\prod_{J\in D}\alpha(\mathbb D_J),\ \alpha(\mathbb D_J)=\beta(\mathbb D_J)\ for\ every\ J\in  D \ and$$
$$q=\frac{n}{\mathrm{lct}(\mathfrak m_\mathbb D)}=w(\{1\})=\cdots=w(\{n\}).$$
\end{enumerate}
\end{thm}

\begin{proof}
This theorem follows from Proposition \ref{Key Prop 1} and Proposition \ref{Key Prop 2}
\end{proof}

\noindent
{\it Acknowledgments.}
The author would like to thank Professor  Shihoko Ishii  and Professor  Shunsuke Takagi for valuable conversations.
The author is partially supported by JSPS Grant-in-Aid for Early-Career Scientists 19K14496 and the
Iwanami Fujukai Foundation.



\end{document}